\documentclass[11pt]{article}
\usepackage[margin=1in]{geometry}
\usepackage[T1]{fontenc}
\usepackage{lmodern}
\usepackage{amsmath,amssymb,amsthm,mathtools}
\usepackage{microtype}
\usepackage{tikz}
\usetikzlibrary{arrows.meta}
\usepackage{enumitem}
\usepackage[colorlinks=true,linkcolor=blue,citecolor=blue,urlcolor=blue]{hyperref}
\hypersetup{
 pdftitle={Global forward self-similar solutions of the spatially inhomogeneous Boltzmann equation for cutoff Maxwell molecules},
 pdfauthor={Quoc-Hung Nguyen; Jiaqi Yang; Tong Yang},
 pdfsubject={Singular homogeneous initial data and forward self-similar mild solutions},
 pdfkeywords={Boltzmann equation, Maxwell molecules, self-similar solutions, singular initial data, Kaniel-Shinbrot iteration}}
\newtheorem{theorem}{Theorem}[section]
\newtheorem{proposition}[theorem]{Proposition}
\newtheorem{lemma}[theorem]{Lemma}
\newtheorem{corollary}[theorem]{Corollary}
\newtheorem{remark}[theorem]{Remark}
\theoremstyle{definition}
\newtheorem{definition}[theorem]{Definition}
\newcommand{\R}{\mathbb R}
\newcommand{\Sph}{\mathbb S}
\newcommand{\Qp}{Q^+}
\newcommand{\dd}{\,\mathrm d}
\newcommand{\eps}{\varepsilon}
\newcommand{\bracket}[1]{\left\langle #1\right\rangle}

\newcommand{\esssup}{\operatorname*{ess\,sup}}
\title{Global forward self-similar solutions of the\\
spatially inhomogeneous Boltzmann equation\\
for cutoff Maxwell molecules}
\author{Quoc-Hung Nguyen \and
Jiaqi Yang \and
Tong Yang}
\date{}
\begin{document}
\maketitle
\begin{abstract}
For every fixed $m>1$, we construct nonzero global forward self-similar
mild solutions of the three-dimensional spatially inhomogeneous Boltzmann
equation for Maxwell molecules. We assume that the angular kernel is
bounded and satisfies a weighted integrability condition at the grazing
and head-on angles. The initial data are homogeneous under the kinetic
scaling, singular at the origin in phase space, and have infinite mass.
Their size is small in a weighted norm invariant under this scaling.
We first prove a bilinear estimate for the gain operator in this norm,
using free-transport coordinates and the orthogonality of the collision
displacements. This estimate gives a global solution of the equation
without the loss term, which bounds the lower and upper approximations
in the Kaniel--Shinbrot iteration. We prove that their difference tends
to zero to arbitrarily high order at the initial time on a fixed annulus
in phase space. The argument uses a finite number of iterations of the
collision estimate and stops characteristics when they leave a spatial
annulus. Self-similarity then gives finite mass and a finite first
absolute velocity moment for the difference. Comparing its mass balance
with its scaling law shows that the two limits agree. The resulting
solution attains its initial trace locally in $L^\infty$ away from the
origin in phase space and has infinite mass at every positive time.
\end{abstract}
\section{Introduction}
We study global forward self-similar solutions of the spatially inhomogeneous
Boltzmann equation. For a fixed $m>1$, the initial data are singular at
$(x,v)=(0,0)$ and satisfy
\[
 f_0(r^m x,rv)=r^{-(m+2)}f_0(x,v),\qquad r>0.
\]
Let $f=f(t,x,v)\ge0$ denote the particle density. In three dimensions, the
equation takes the form
\begin{equation}
 \partial_t f+v\cdot\nabla_x f=Q(f,f),
 \qquad (t,x,v)\in(0,\infty)\times\R^3\times\R^3.
 \label{eq:Boltzmann-intro}
\end{equation}
For Maxwell molecules, the collision operator in the sigma representation is
\begin{align}
 Q(F,G)(v)&=\Qp(F,G)(v)-\nu_G F(v),
 \label{eq:Q-decomposition-intro}\\
 \Qp(F,G)(v)&=\int_{\R^3}\int_{\Sph^2}
 \mathbf{b}(\widehat{v-v_*}\cdot\sigma)F(v')G(v_*')\,\dd\sigma\dd v_*,
 \label{eq:gain-intro}
\end{align}
where
\begin{equation}
 v'=\frac{v+v_*}{2}+\frac{|v-v_*|}{2}\sigma,
 \qquad
 v_*'=\frac{v+v_*}{2}-\frac{|v-v_*|}{2}\sigma.
 \label{eq:collision-rule-intro}
\end{equation}
Here $\widehat z=z/|z|$ for $z\ne0$; the value of the integrand on
$v=v_*$ is immaterial. Throughout the paper we assume
\begin{equation}
 \begin{gathered}
 0\le \mathbf{b}\in L^\infty([-1,1]),\\
 \Lambda_{\mathbf{b}}:=\int_{-1}^1\mathbf{b}(z)
 \bigl[(1-z)^{-3/2}+(1+z)^{-3/2}\bigr]\,\dd z<\infty.
 \end{gathered}
 \label{eq:strong-cutoff-intro}
\end{equation}
This condition allows the angular support to reach both endpoints.
For example, it holds if
$0\le\mathbf{b}(z)\le C(1-z^2)^\alpha$ with $\alpha>1/2$.
Constants depending on $\mathbf{b}$ depend only on
$\|\mathbf{b}\|_{L^\infty}$ and $\Lambda_{\mathbf{b}}$. With
\[
 \beta_{\mathbf{b}}:=2\pi\int_{-1}^1\mathbf{b}(z)\,\dd z,
 \qquad \rho_G(t,x):=\int_{\R^3}G(t,x,v)\,\dd v,
\]
the collision frequency is independent of velocity:
\begin{equation}
 \nu_G(t,x)=\beta_{\mathbf{b}}\rho_G(t,x).
 \label{eq:maxwell-frequency-intro}
\end{equation}
The assumption on $\Lambda_{\mathbf{b}}$ ensures that the constants in
Lemma~\ref{lem:angular-integration} are finite. We use these angular identities
to estimate velocity integrals after localizing in space. We also use
\eqref{eq:maxwell-frequency-intro} when integrating the equation for the
difference between the lower and upper limits.

Global solutions of the cutoff Boltzmann equation have been studied
under several assumptions on the initial data. For background on the
Boltzmann equation, see \cite{CercignaniIllnerPulvirenti}. For finite-mass
data with suitable entropy and moment bounds,
DiPerna--Lions \cite{DiPernaLions} established global renormalized
solutions. Spectral and energy methods give global solutions near a
Maxwellian in periodic domains and in the whole space
\cite{Ukai1974,GuoCutoff,LiuYangYu,Duan2008}. Related theories in spatial
Besov spaces are developed in \cite{DuanLiuXu,DuanSakamoto}.
Large-amplitude perturbations can also be treated under integral
smallness conditions \cite{DuanHuangWangYang}; bounded-domain results
with diffuse or specular reflection are obtained in
\cite{DuanWang,DuanKoLee}. For cutoff Maxwell molecules,
Duan--Liu--Yang \cite{DuanLiuYangCouette} constructed stationary plane
Couette flows at small shear rates and proved stability within the
planar class. For the sheared cutoff Maxwell equation on the
three-dimensional torus, Duan--Liu--Shen \cite{DuanLiuShenShear} proved
global stability of homogeneous self-similar profiles in the rescaled
variables under small spatially periodic perturbations with matching
mass and momentum, for sufficiently small shear rates.

For data close to vacuum, the construction of Illner--Shinbrot
\cite{IllnerShinbrot} uses smallness and decay along free transport,
and the monotone iteration of Kaniel--Shinbrot \cite{KanielShinbrot}
preserves nonnegativity. More directly related to our method,
Chen--Denlinger--Pavlovi\'c \cite{ChenDenlingerPavlovic} proved a
critical small-data theory for the two-dimensional constant-kernel
equation, with additional regularity and localization assumptions.
In their Section~2.2, a solution of the equation without the loss term
provides the upper bound used to start Kaniel--Shinbrot iteration.
Chen--Shen--Zhang \cite{ChenShenZhang}
used this framework for the three-dimensional cutoff equation with
kinetic factor $|v-v_*|^\gamma$ for $-\tfrac12\leq\gamma\leq0$.
Their solutions belong to weighted
$L^2_vH^s_x$ spaces with $s>1$, and the smallness assumption is imposed in
a lower scaling-critical norm. We also first solve the equation without
the loss term and then apply Kaniel--Shinbrot iteration. For the singular
homogeneous data considered here, we need different estimates to construct
the approximations and prove that their limits agree.
The equation without the loss term can blow up outside a small-data regime
\cite{AndreassonCalogeroIllner}.

For long-range collision models, global stability of vacuum has been
proved for small initial data with sufficient regularity, spatial
localization, and velocity decay. Luk \cite{LukVacuum} treated the
spatially inhomogeneous Landau equation on $\R^3$ with moderately soft
potentials, $-2<\gamma<0$. Chaturvedi \cite{ChaturvediBoltzmannVacuum}
proved a corresponding result for the non-cutoff Boltzmann equation
with $0<s<1$ and $0<\gamma+2s<2$, where $s$ is the angular singularity
parameter. Chaturvedi \cite{ChaturvediLandauHardVacuum} also established
stability of vacuum for the Landau equation with Maxwellian and hard
potentials, $0\leq\gamma<1$. These arguments use dispersion from free
transport together with weighted estimates for the collision operator.
For the Coulomb Landau equation, Chaturvedi--Luk
\cite{ChaturvediLukLinear} proved linear stability of traveling
Maxwellians. For very soft potentials, including the Coulomb Landau equation, the
corresponding global nonlinear stability problem near vacuum remains,
to our knowledge, an important open problem.

Local well-posedness and ill-posedness in Sobolev spaces have also been
studied. For three-dimensional cutoff kernels with kinetic factor
$|v-v_*|^\gamma$, $-1\leq\gamma\leq0$, Chen--Shen--Zhang
\cite{ChenShenZhangSoft} proved local well-posedness in $L^2_vH^s_x$
with suitable polynomial velocity weights for $s>1$, and failure of
uniform continuity of the solution map in the corresponding weighted
spaces for $0\leq s<1$. For the hard-sphere equation,
Chen--Guo--Shen--Zhang \cite{ChenGuoShenZhangHardSphere} proved failure
of uniform continuity for $0\leq s<1$, even when the initial data are
close in a Gaussian-weighted $L^2_vH^s_x$ norm and the solutions are
compared in the unweighted $L^2_vH^s_x$ norm.

For the spatially homogeneous Maxwell equation, Cannone--Karch
\cite{CannoneKarch} studied self-similar solutions and infinite-energy
asymptotics. Self-similar profiles for homoenergetic solutions of the form
$f(t,x,v)=g(t,v-L(t)x)$ were constructed in
\cite{JamesNotaVelazquez,Kepka}. In these results, the ansatz fixes the spatial dependence of the solution.
Here we allow the initial trace to depend separately on position and
velocity through a bounded profile with compact velocity support.

Singular homogeneous data also arise in self-similar constructions for
fluid equations. Jia--\v Sver\'ak \cite{JiaSverak} treated
$(-1)$-homogeneous Navier--Stokes data without an amplitude smallness
assumption. For Muskat, small corners and cones were studied in
\cite{GarciaGomezNguyenPausader,NaMuskat}, with moving corners treated in
\cite{GarciaGomezHaziotPausader}. Surface-tension problems with a different
similarity scale were studied for Hele--Shaw in \cite{AgrawalPatel} and
for Muskat in \cite{WanYang}. In these constructions, estimates near the
initial time help identify the prescribed trace through the behavior of
the profile at infinity.

Our initial data depend on both position and velocity. They are unbounded
near the origin in phase space and have infinite mass. We therefore
measure their size in a weighted norm that is invariant under the kinetic
scaling. These data do not belong to the bounded perturbation classes or
the global $L^2$ classes described above, as the calculation below shows. The works \cite{BedrossianGualdaniSnelson} and
\cite{HendersonSnelsonTarfulea} concern, respectively, backward
singularity scenarios and rough-data non-cutoff solutions. Throughout this paper
we assume the angular integrability condition \eqref{eq:strong-cutoff-intro}.

Fix $m>1$. We first recall the two-parameter dilation symmetry of the
Maxwell-molecule equation. For $\lambda,\mu>0$, define
\[
 f_{\lambda,\mu}(t,x,v)
 :=\lambda\mu^3 f(\lambda t,\lambda\mu x,\mu v).
\]
Writing
$(T,X,V)=(\lambda t,\lambda\mu x,\mu v)$, direct differentiation and the
change of variables $V_*=\mu v_*$ in the collision integral give
\begin{align*}
 (\partial_t+v\cdot\nabla_x)f_{\lambda,\mu}(t,x,v)
 &=\lambda^2\mu^3
 (\partial_T+V\cdot\nabla_X)f(T,X,V),\\
 Q(f_{\lambda,\mu},f_{\lambda,\mu})(t,x,v)
 &=\lambda^2\mu^3Q(f,f)(T,X,V).
\end{align*}
The two factors of $f$ contribute $\lambda^2\mu^6$, while
$\dd v_*=\mu^{-3}\dd V_*$. The angular kernel is unchanged because
$\widehat{\mu v-\mu v_*}=\widehat{v-v_*}$. Thus
$f_{\lambda,\mu}$ is a solution whenever $f$ is a solution.

We now choose the one-parameter subgroup
$\lambda=r^{m-1}$ and $\mu=r$. It gives
\begin{equation}
 f(t,x,v)\longmapsto
 r^{m+2}f(r^{m-1}t,r^mx,rv),\qquad r>0.
 \label{eq:kinetic-scaling-intro}
\end{equation}
This choice matches the initial trace below, since
\[
 f_0(r^mx,rv)=r^{-(m+2)}f_0(x,v).
\]
Consequently, the transformation in
\eqref{eq:kinetic-scaling-intro} leaves $f_0$ unchanged. If a solution is
invariant under this transformation, choosing
$r=t^{-1/(m-1)}$ and setting $F(X,V)=f(1,X,V)$ gives
\[
 f(t,x,v)=t^{-\frac{m+2}{m-1}}
 F\left(\frac{x}{t^{m/(m-1)}},
         \frac{v}{t^{1/(m-1)}}\right).
\]
The condition $m>1$ makes an angular integral in the gain estimate finite.
This estimate allows us to solve the equation without the loss term and
start the Kaniel--Shinbrot iteration. The resulting lower and upper limits
satisfy $f^{\mathrm{lo}}\le f^{\mathrm{up}}$.
The main difficulty is to prove that they agree. We estimate their
difference $w=f^{\mathrm{up}}-f^{\mathrm{lo}}$ as $t\to0$ on a fixed annulus
in phase space. Self-similarity then shows that $w$ has finite mass.
Comparing the mass balance with the scaling law gives $w=0$, and hence
$f^{\mathrm{lo}}=f^{\mathrm{up}}$. We explain these steps below.
\subsection{Main result}
We now state the main theorem and specify the admissible initial traces.
Fix $R_0<\infty$ and a nonzero, nonnegative, bounded measurable profile
$\phi$ on $\Sph^2\times\R^3$. We choose a representative that vanishes
when $|z|>R_0$. For $\eps>0$, we prescribe the initial trace
\begin{equation}
 f_0(x,v)=\eps|x|^{-(m+2)/m}
 \phi\left(\widehat x,\frac{v}{|x|^{1/m}}\right),\qquad x\ne0,
 \label{eq:trace-intro}
\end{equation}
where
\begin{equation}
 0\le\phi\in L^\infty(\Sph^2\times\R^3),\qquad
 \phi\not\equiv0,\qquad
 \phi(\omega,z)=0\quad\text{when $|z|>R_0$}.
 \label{eq:phi-intro}
\end{equation}
We set $f_0(0,v)=0$. The velocity-support condition on $\phi$ gives
$|v|\lesssim |x|^{1/m}$ on the support of $f_0$. In particular, $f_0$ is
measurable and locally bounded away from the kinetic origin
$(x,v)=(0,0)$. Its velocity density
is a nonnegative angular coefficient times $|x|^{-(m-1)/m}$, so the total
mass is infinite. The parameter $\eps$ measures the size of the trace
without removing its singularity.

More precisely, let $R>1$ and
$1\le p<3(m+1)/(m+2)$. Under the substitutions
$x=r\omega$ and $v=r^{1/m}z$, the restrictions $|x|<R$ and
$|v|<R^{1/m}$ become $0<r<R/\max\{1,|z|\}^m$.
Tonelli's theorem therefore gives
\[
 \int_{\R^6}\mathbf{1}_{\{|x|<R,\,|v|<R^{1/m}\}}
 |f_0(x,v)|^p\,\dd x\dd v
 =\frac{mC_{\phi,p}\eps^p}{3m+3-p(m+2)}
 R^{\frac{3m+3-p(m+2)}m}<\infty,
\]
where
\[
 C_{\phi,p}:=\int_{\Sph^2}\int_{\R^3}|\phi(\omega,z)|^p
 \max\{1,|z|\}^{p(m+2)-3m-3}\,\dd z\dd\omega\in(0,\infty).
\]
For $3(m+1)/(m+2)\le p<\infty$, the radial integral diverges at zero,
logarithmically at the endpoint. Since $\phi\not\equiv0$, local
integrability fails at the kinetic origin in this range.
Thus $f_0\in L^p_{\mathrm{loc}}(\R^6)$ precisely for
$1\le p<3(m+1)/(m+2)$.
For these $p$, the displayed integral diverges as $R\to\infty$; for all
other $p\ge1$, local integrability fails. Hence
$f_0\notin L^p(\R^6)$ for any $1\le p<\infty$, although
$f_0\in L^2_{\mathrm{loc}}(\R^6)$. Thus the weighted $L^2$ theory in
\cite{ChenShenZhang} does not apply.

To describe the solution along free transport, we introduce
\begin{equation}
 h^\#(t,y,v):=h(t,y+tv,v),\qquad
 \varrho_m(y,v):=(|y|^2+|v|^{2m})^{1/(2m)}.
 \label{eq:anisotropic-radius-intro}
\end{equation}
Since $\varrho_m(r^my,rv)=r\varrho_m(y,v)$, we use the weight
$\varrho_m^{m+2}$ and the norm
\[
 \|h\|_{Y_m}:=\esssup_{t>0,y,v}
 \varrho_m(y,v)^{m+2}|h^\#(t,y,v)|.
\]
This norm is invariant under \eqref{eq:kinetic-scaling-intro}.
We write $\bracket z=(1+|z|^2)^{1/2}$.
We formulate the equation through the characteristic integral identity.
For $0\le s\le t$, write
\[
 A_h(s,t;y,v):=\int_s^t\nu_h(\tau,y+\tau v)\,\dd\tau.
\]
Here $\nu_h$ is the collision frequency defined in
\eqref{eq:maxwell-frequency-intro}.
\begin{definition}[Mild solution]
A nonnegative measurable function $f$ is a mild solution of
\eqref{eq:Boltzmann-intro} with initial trace $f_0$ if its velocity density
is finite almost everywhere, if for almost every $(y,v)$ and every
$0<T<\infty$,
\[
 A_f(0,T;y,v)<\infty,
 \qquad
 \int_0^T\Qp(f,f)(s,y+sv,v)\,\dd s<\infty,
\]
and if
\begin{equation}
 f^\#(t,y,v)=e^{-A_f(0,t;y,v)}f_0(y,v)
 +\int_0^t e^{-A_f(s,t;y,v)}\Qp(f,f)(s,y+sv,v)\,\dd s
 \label{eq:mild-definition}
\end{equation}
for almost every $(t,y,v)$. These integrability conditions give an
absolutely continuous representative along almost every characteristic,
with $f^\#(t,y,v)\to f_0(y,v)$ as $t\downarrow0$.
\end{definition}
The main result is the following.
\begin{theorem}[Small forward self-similar solutions]
\label{thm:main}
Let $m>1$, and assume \eqref{eq:strong-cutoff-intro} and
\eqref{eq:phi-intro}. There exists
$\eps_0=\eps_0(m,\mathbf{b},\phi)>0$ such that, for every
$0<\eps\le\eps_0$, equation~\eqref{eq:Boltzmann-intro} has a
nonnegative, nonzero, global mild solution with trace \eqref{eq:trace-intro}.
It has a characteristic representative with the following properties.
\begin{enumerate}[label=\textup{(\roman*)}]
\item For every $r>0$, it is invariant under the kinetic dilation:
\begin{equation}
 f(t,x,v)=r^{m+2}f(r^{m-1}t,r^mx,rv).
 \label{eq:main-selfsimilarity}
\end{equation}
\item For a constant $C$ depending only on $m,\mathbf{b},\phi$,
\begin{align}
 0\le f^\#(t,y,v)&\le C\eps\varrho_m(y,v)^{-(m+2)},
 \label{eq:main-envelope}\\
 \rho_f(t,x)&\le C\eps t^{-1}
 \bracket{\frac{x}{t^{m/(m-1)}}}^{-(m-1)/m}.
 \label{eq:main-density}
\end{align}
\item The initial trace is attained locally in $L^\infty$ away from the origin:
\begin{equation}
 f^\#(t,\cdot,\cdot)\longrightarrow f_0
 \quad\text{in }L^\infty_{\mathrm{loc}}(\R^6\setminus\{(0,0)\})
 \quad(t\downarrow0).
 \label{eq:main-trace}
\end{equation}
\item For every $t>0$, one has $\int_{\R^6}f(t,x,v)\,\dd x\dd v=\infty$.
\end{enumerate}
The characteristic identities and bounds are understood away from the
kinetic-origin characteristic; its values may be set to zero.
\end{theorem}
The solution has finite velocity density and infinite total mass at
every positive time.
The bound in \eqref{eq:main-envelope} is measured at the free spatial
foot $y=x-tv$, and the convergence in \eqref{eq:main-trace} identifies
the trace away from the singular characteristic.
The theorem does not assert uniqueness among all mild solutions with
the same trace. We prove below that the small solution of the equation
without the loss term is unique in the stated ball and that the
constructed brackets coincide.
\subsection{Main ideas of the proof}
We construct an increasing sequence of lower approximations and a
decreasing sequence of upper approximations. To obtain a solution, we
must prove that their limits agree. Neither limit is known to have finite
mass at this stage. We instead prove that their difference has finite
mass, and then use the mass balance and self-similarity to show that the
difference is zero.

\paragraph{1. A solution of the equation without the loss term.}
We first construct a nonnegative function $U$ satisfying
\[
 U^\#(t,y,v)=f_0(y,v)
 +\int_0^t\Qp(U,U)(s,y+sv,v)\,\dd s.
\]
This is the characteristic form of the equation with the loss term
removed. The bilinear estimate in Proposition~\ref{prop:critical-gain}
controls the integral in the norm $Y_m$ by
$C_{m,\mathbf{b}}\|U\|_{Y_m}^2$.
For small $\eps$, a contraction argument therefore gives a global
solution with $\|U\|_{Y_m}\le 2C_{m,\phi}\eps$, where
$\|f_0(x-tv,v)\|_{Y_m}\le C_{m,\phi}\eps$.
All the lower and upper approximations will be bounded by $U$.

To prove the gain estimate, we use the Carleman representation with
$a=v-v'$, $c=v-v_*'$, and $a\cdot c=0$. This orthogonality allows us to
bound the product of the two weights by a sum of two terms. We integrate
one variable over the plane $a^\perp$, with measure
$\dd\mathcal H^2(c)$, and then integrate the remaining weight along a
line. The resulting integrals are estimated in
\eqref{eq:carleman-plane-bound}--\eqref{eq:carleman-final-integral}.

\paragraph{2. Lower and upper approximations.}
The gain term increases when its input increases, but the exponential
loss factor decreases. Indeed, if $g\le h$, then
$\Qp(g,g)\le\Qp(h,h)$, whereas $\nu_g\le\nu_h$ gives
$e^{-A_g}\ge e^{-A_h}$. For this reason, a single Picard sequence need
not be monotone. We use two sequences, choosing the gain and collision
frequency separately to preserve the lower and upper bounds.

Starting from $f_0^{\mathrm{lo}}=0$ and $f_0^{\mathrm{up}}=U$, we use the
Kaniel--Shinbrot iteration to construct lower approximations
$f_n^{\mathrm{lo}}$ and upper approximations $f_n^{\mathrm{up}}$.
For $f_{n+1}^{\mathrm{lo}}$, we use the smaller gain
$\Qp(f_n^{\mathrm{lo}},f_n^{\mathrm{lo}})$ and the larger collision
frequency $\nu_{f_n^{\mathrm{up}}}$. For $f_{n+1}^{\mathrm{up}}$, we use
the larger gain $\Qp(f_n^{\mathrm{up}},f_n^{\mathrm{up}})$ and the smaller
collision frequency $\nu_{f_n^{\mathrm{lo}}}$; see
\eqref{eq:lower-iterate}--\eqref{eq:upper-iterate}.
Each step solves a linear transport equation with nonnegative source
and damping, and hence preserves nonnegativity. These choices make the
lower sequence increase and the upper sequence decrease, while both
remain bounded by $U$. More precisely,
\[
 0\le f_n^{\mathrm{lo}}\le f_{n+1}^{\mathrm{lo}}
 \le f_{n+1}^{\mathrm{up}}\le f_n^{\mathrm{up}}\le U.
\]
Thus the limits
\[
 f^{\mathrm{lo}}:=\lim_{n\to\infty}f_n^{\mathrm{lo}},\qquad
 f^{\mathrm{up}}:=\lim_{n\to\infty}f_n^{\mathrm{up}}
\]
exist and satisfy $0\le f^{\mathrm{lo}}\le f^{\mathrm{up}}\le U$.
Both limits have the prescribed trace and satisfy the kinetic scaling.
They solve a coupled system: the equation for $f^{\mathrm{lo}}$ contains
the collision frequency $\nu_{f^{\mathrm{up}}}$, and the equation for
$f^{\mathrm{up}}$ contains $\nu_{f^{\mathrm{lo}}}$. Thus, once we prove
$f^{\mathrm{lo}}=f^{\mathrm{up}}$, their common value solves the full
Boltzmann equation.

\paragraph{3. A short-time estimate for the difference.}
Set $w:=f^{\mathrm{up}}-f^{\mathrm{lo}}\ge0$. We estimate $w$ on the fixed
annulus
\[
 \mathcal A_m:=\{(y,v):1/2\le\varrho_m(y,v)\le2\}.
\]
The limits $f^{\mathrm{lo}}$ and $f^{\mathrm{up}}$ have the same initial
trace. However, the bound for the collision frequency is proportional
to $\eps/t$, which is not integrable at time zero. We therefore need
an estimate for the rate at which $w$ tends to zero as $t\downarrow0$.

We divide $\mathcal A_m$ into a region near $y=0$ and a region away from
$y=0$. In the first region, $|v|$ is bounded below, so the support
condition $|v|\lesssim |y|^{1/m}$ makes the initial trace zero.
We apply the high-velocity estimate to the integral formula for $U$
and repeat this estimate a fixed number of times. At each collision,
energy conservation ensures that
at least one incoming velocity remains large enough to apply the
high-velocity estimate. For sufficiently small regions and times, the
support condition makes all the free-transport terms in this iteration
zero. The density integral along each characteristic gives a factor
$t^{1/m}$. Repeating the estimate gives a high power of $t$ for $U$,
and therefore for $w\le U$.

Away from $y=0$, subtracting the two limit equations gives a linear
collision inequality for $w$ with a positive right-hand side. We iterate
this inequality along characteristics and stop each characteristic when
it leaves the spatial region. Since the successive regions are separated
by a positive distance, a characteristic that exits in a short time has
large velocity. We bound its contribution by rescaling the estimate near
$y=0$.
Choose an integer
\[
 N_m>\frac{2m+2}{m-1}.
\]
Combining the two regions gives
\[
 \esssup_{(y,v)\in\mathcal A_m}w^\#(t,y,v)\le C_mt^{N_m}
\]
for sufficiently small $t>0$; see
Proposition~\ref{prop:finite-gap-main}.

\paragraph{4. Scaling and the mass of the difference.}
Self-similarity converts the short-time estimate into decay at large
positions and velocities. For each fixed large scale $R$ and almost
every $(y,v)\in\mathcal A_m$, we have
\[
 w^\#(1,R^my,Rv)=R^{-(m+2)}w^\#(R^{-(m-1)},y,v)
 \le C_mR^{-m-2-(m-1)N_m}.
\]
Using dyadic scales, we obtain
$w^\#(1,y,v)\le
C_m\varrho_m(y,v)^{-m-2-(m-1)N_m}$ for large $\varrho_m(y,v)$.
The dilation $(y,v)\mapsto(R^my,Rv)$ multiplies volume by $R^{3m+3}$,
and $|v|\le\varrho_m(y,v)$. The choice of $N_m$ makes this decay, together
with the bound $w^\#\le C_m\eps\varrho_m^{-m-2}$ near the origin,
integrable at
time one. Scaling then gives
\[
 \int_{\R^6}\langle v\rangle w(t,x,v)\,\dd x\dd v<\infty
 \qquad(t>0).
\]
This explains the choice of $N_m$: it makes
both the mass and the first absolute velocity moment of $w$ finite.

We can now integrate the equation for $w$ over space and velocity.
Write
\[
 M(t):=\int_{\R^6}w(t,x,v)\,\dd x\dd v.
\]
The finite first velocity moment allows us to remove the spatial cutoff
in this calculation. The Maxwell collision identity and the density
bound $\rho_U(t,x)\le A_{\rho,m}\eps/t$ from
\eqref{eq:density-envelope} give
\[
 M'(t)\le\frac{2\beta_{\mathbf{b}}A_{\rho,m}\eps}{t}M(t).
\]
On the other hand, the exact scaling of $w$ gives
\[
 M(t)=t^{\frac{2m+1}{m-1}}M(1),\qquad
 M'(t)=\frac{2m+1}{(m-1)t}M(t).
\]
Here the change of variables in space and velocity contributes
$t^{(3m+3)/(m-1)}$, while the amplitude contributes
$t^{-(m+2)/(m-1)}$. If
$2\beta_{\mathbf{b}}A_{\rho,m}\eps<(2m+1)/(m-1)$, the two relations for
$M'$ force $M=0$.
Since $w\ge0$, this proves $w=0$ and hence
$f^{\mathrm{lo}}=f^{\mathrm{up}}$. Their common value
$f$ is the required solution. Its bounds and scaling follow from the
construction, and the localization estimates give the stated initial
trace.

\paragraph{On the smallness assumption.}
We do not know whether Theorem~\ref{thm:main} holds for every $\eps>0$
when the kernel is nontrivial and the profile $\phi$ is fixed. Since the
kinetic dilation leaves the homogeneous trace unchanged, it cannot
reduce the amplitude of the data. Our proof uses smallness in two places: the construction of $U$
requires $4C_{m,\mathbf{b}}C_{m,\phi}\eps<1$, and the mass argument requires
$2\beta_{\mathbf{b}}A_{\rho,m}\eps<(2m+1)/(m-1)$. A higher-order
short-time estimate would improve the decay of $w$, but would not change
the mass-scaling exponent. It therefore
would not remove the second restriction by itself.

These restrictions do not show that solutions fail to exist for large
amplitudes. Duan--Huang--Wang--Yang \cite{DuanHuangWangYang} allow
perturbations of a global Maxwellian with large weighted
$L^\infty_{x,v}$ amplitude, while requiring small relative
entropy and a small $L_x^1L_v^\infty$ norm of the normalized
perturbation. Their result therefore retains a smallness
assumption in these integral quantities; it is not a global
existence theorem for arbitrary large data. Moreover, its
weighted boundedness requirement excludes the singular
initial trace considered here. Likewise, blowup for the equation
without the loss term
\cite{AndreassonCalogeroIllner} does not imply blowup for the full
equation, which contains the loss term.
The work of Jia and \v Sver\'ak \cite{JiaSverak} suggests a possible
approach through a priori estimates and compactness or degree theory.
To apply such an approach here, one would need estimates for the collision
density and the homogeneous far trace without a smallness assumption.
One would also need to control the nonlinear collision term when passing
to a limit. These estimates may require keeping the loss term. We do not
prove them here.

\paragraph{Organization of the paper.}
Section~\ref{sec:preliminaries} records the scaling, collision geometry,
and properties of the trace. Section~\ref{sec:gain} proves the gain
estimate and constructs $U$. Section~\ref{sec:KS} constructs the lower
and upper approximations, and Section~\ref{sec:gap-flatness} proves the
short-time estimate for their difference. Section~\ref{sec:closure}
proves that the limits agree and completes the proof of the main theorem.
Section~\ref{sec:profile} discusses the self-similar profile.
\section{Scaling, collision geometry, and the trace}
\label{sec:preliminaries}
We record the collision identities and the scaling properties of the
initial trace. The orthogonality of the collision displacements will be
used in the gain estimate.

The collision rule \eqref{eq:collision-rule-intro} preserves momentum and
energy.  Set
\[
 q:=v-v_*,
 \qquad
 a:=v-v'=\frac{q-|q|\sigma}{2},
 \qquad
 c:=v-v_*'=\frac{q+|q|\sigma}{2}.
\]
Then
\begin{equation}
 a\cdot c=0,
 \qquad
 |a|^2=\frac{|q|^2}{2}(1-\widehat q\cdot\sigma),
 \qquad
 |c|^2=\frac{|q|^2}{2}(1+\widehat q\cdot\sigma).
 \label{eq:orthogonal-displacements}
\end{equation}
Consequently,
\begin{equation}
 |a|^2+|c|^2=|q|^2,\qquad |a|,|c|\leq |q|.
 \label{eq:displacement-comparability}
\end{equation}
Figure~\ref{fig:collision-geometry} illustrates these identities.
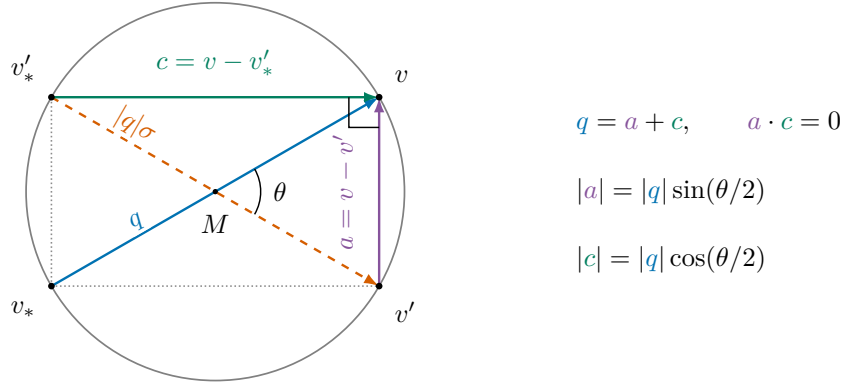
\begin{figure}[htbp]
\centering
\begin{tikzpicture}[x=1cm,y=1cm,>={Latex[length=2mm,width=1.4mm]},font=\small,
  vector/.style={->,line width=0.9pt},
  guide/.style={draw=black!45,densely dotted,line width=0.6pt}]
 \definecolor{collisionQ}{HTML}{0072B2}
 \definecolor{collisionSigma}{HTML}{D55E00}
 \definecolor{collisionA}{HTML}{87549C}
 \definecolor{collisionC}{HTML}{007F5F}
 \coordinate (vs) at (0,0);
 \coordinate (vp) at (4.330127,0);
 \coordinate (vsp) at (0,2.5);
 \coordinate (v) at (4.330127,2.5);
 \coordinate (M) at (2.165064,1.25);
 \draw[draw=black!50,line width=0.65pt] (M) circle[radius=2.5];
 \draw[guide] (vs)--(vp);
 \draw[guide] (vs)--(vsp);
 \draw[vector,draw=collisionQ] (vs)--(v)
   node[pos=0.28,above,sloped,fill=white,inner sep=1.2pt,text=collisionQ] {$q$};
 \draw[vector,dashed,draw=collisionSigma] (vsp)--(vp)
   node[pos=0.23,above,sloped,fill=white,inner sep=1.2pt,text=collisionSigma]
   {$|q|\sigma$};
 \draw[vector,draw=collisionC] (vsp)--(v) node[midway,above=4pt,text=collisionC] {$c=v-v_*'$};
 \draw[vector,draw=collisionA] (vp)--(v) node[midway,sloped,above=5pt,text=collisionA] {$a=v-v'$};
 \draw[line width=0.55pt] (3.930127,2.5)--(3.930127,2.1)--(4.330127,2.1);
 \draw[line width=0.55pt] (M) ++(-30:0.60)
   arc[start angle=-30,end angle=30,radius=0.60];
 \node at (3.02,1.25) {$\theta$};
 \fill (M) circle (1pt) node[below=5pt] {$M$};
 \fill (vs) circle (1.25pt) node[below left=2pt] {$v_*$};
 \fill (vp) circle (1.25pt) node[below right=2pt] {$v'$};
 \fill (vsp) circle (1.25pt) node[above left=2pt] {$v_*'$};
 \fill (v) circle (1.25pt) node[above right=2pt] {$v$};
 \node[anchor=west] at (6.8,2.15) {${\color{collisionQ}q}={\color{collisionA}a}+{\color{collisionC}c},\qquad {\color{collisionA}a}\cdot{\color{collisionC}c}=0$};
 \node[anchor=west] at (6.8,1.25) {$|{\color{collisionA}a}|=|{\color{collisionQ}q}|\sin(\theta/2)$};
 \node[anchor=west] at (6.8,0.35) {$|{\color{collisionC}c}|=|{\color{collisionQ}q}|\cos(\theta/2)$};
\end{tikzpicture}
\caption{A planar view of a nondegenerate collision in velocity space.
The four velocities lie on the circle with center $M=(v+v_*)/2$
and radius $|q|/2$. The vectors $a$ and $c$ form perpendicular sides of a rectangle,
with $q=a+c$. Here
$\theta=\arccos(\widehat q\cdot\sigma)$.}
\label{fig:collision-geometry}
\end{figure}
Either displacement can be arbitrarily small relative to $|q|$.
We use orthogonality in the gain estimate and the angular integration
identities of Lemma~\ref{lem:angular-integration} in the localization estimates.
The trace satisfies
\begin{equation}
 f_0(r^mx,rv)=r^{-(m+2)}f_0(x,v).
 \label{eq:trace-homogeneity}
\end{equation}
The support condition in \eqref{eq:phi-intro} gives
\begin{equation}
 f_0(x,v)\neq0\quad\Longrightarrow\quad
 |v|^m\leq R_0^m|x|.
 \label{eq:conic-support-main}
\end{equation}
In particular,
\begin{equation}
 0\leq f_0(y,v)\leq C_{m,\phi}\eps
 \varrho_m(y,v)^{-(m+2)}.
 \label{eq:trace-weight-bound}
\end{equation}
Indeed, on the support of $f_0$ one has
$\varrho_m(y,v)\leq C_{m,R_0}|y|^{1/m}$, while outside that support the claim is
trivial.
If an exactly self-similar solution exists and
\[
 F(X,V):=f(1,X,V),
\]
then
\begin{equation}
 f(t,x,v)=t^{-\frac{m+2}{m-1}}
 F\left(\frac{x}{t^{m/(m-1)}},
 \frac{v}{t^{1/(m-1)}}\right).
 \label{eq:profile-ansatz}
\end{equation}
Formally, and distributionally whenever the terms are defined, $F$ solves
\begin{equation}
 -\frac{m+2}{m-1}F-\frac{m}{m-1}X\cdot\nabla_XF
 -\frac{1}{m-1}V\cdot\nabla_VF
 +V\cdot\nabla_XF=Q(F,F).
 \label{eq:profile-equation}
\end{equation}
We construct $f$ in time-dependent variables and recover the profile and
its far trace in Section~\ref{sec:profile}.
\section{The critical gain estimate and the equation without the loss term}
\label{sec:gain}
We first estimate the gain operator in the scaling-invariant space $Y_m$.
We then use this estimate to solve the equation without the loss term.
Its solution $U$ is the initial upper approximation in
Section~\ref{sec:KS}.

Define
\begin{equation}
 \|h\|_{Y_m}:=\esssup_{t>0,\,y,v}
 \varrho_m(y,v)^{m+2}|h^\#(t,y,v)|.
 \label{eq:Y-norm}
\end{equation}
For nonnegative $F,G$, let
\begin{equation}
 \mathcal B(F,G)^\#(t,y,v)
 :=\int_0^t\Qp(F,G)(s,y+sv,v)\,\dd s.
 \label{eq:B-definition}
\end{equation}
\begin{proposition}[Critical anisotropic gain estimate]
\label{prop:critical-gain}
For $0\le\mathbf{b}\in L^\infty([-1,1])$,
\begin{equation}
 \|\mathcal B(F,G)\|_{Y_m}
 \leq C_m\|\mathbf{b}\|_\infty\|F\|_{Y_m}\|G\|_{Y_m}
 \label{eq:critical-gain}
\end{equation}
for all nonnegative measurable $F,G$ with finite $Y_m$ norm.
\end{proposition}
\subsection{Proof of the critical anisotropic gain estimate}
To prove Proposition~\ref{prop:critical-gain}, we use the classical
Carleman representation for the cutoff gain operator; see
\cite{Wennberg1994,MouhotVillaniCutoff,AlonsoCarneiro}. We apply this
representation to the singular anisotropic weight
\[
 \Phi_m(y,v):=(|y|^2+|v|^{2m})^{-\frac{m+2}{2m}}
 =\varrho_m(y,v)^{-(m+2)}.
\]

\begin{lemma}[Carleman representation and collision maps]
\label{lem:cutoff-collision-coordinates}
For nonnegative measurable $F,G$, one has
\begin{equation}
 \Qp(F,G)(v)
 =4\int_{\R^3}\frac{F(v-a)}{|a|}
 \int_{a^\perp}k_{\mathbf{b}}(a,c)G(v-c)\,\dd\mathcal H^2(c)\dd a,
 \label{eq:carleman-gain}
\end{equation}
where $\mathcal H^2$ denotes two-dimensional Hausdorff measure and
\[
 r_{a,c}:=(|a|^2+|c|^2)^{1/2},\qquad
 k_{\mathbf{b}}(a,c):=\frac{1}{r_{a,c}}
 \mathbf{b}\left(\frac{|c|^2-|a|^2}{|a|^2+|c|^2}\right).
\]
For fixed $\sigma$ and $q\ne0$, the collision maps satisfy
\begin{equation}
 \det D_qa=\frac18(1-\widehat q\cdot\sigma),\qquad
 \det D_qc=\frac18(1+\widehat q\cdot\sigma).
 \label{eq:collision-map-jacobians}
\end{equation}
They are injective where their determinants do not vanish. The degenerate
sets are the rays $\widehat q=\sigma$ and $\widehat q=-\sigma$,
respectively; these have zero Lebesgue measure in $q$.
\end{lemma}
\begin{proof}
For fixed $q\ne0$, the map $\sigma\mapsto a=(q-|q|\sigma)/2$
parametrizes the sphere $|a-q/2|=|q|/2$ with
$\dd\sigma=4|q|^{-2}\dd\mathcal H^2(a)$.
On this sphere, $|\nabla_a(|a|^2-a\cdot q)|=|q|$.
For every nonnegative measurable function $H$, the coarea formula gives
\begin{align*}
 \int_{\R^3}\int_{\Sph^2}H(a,c)\,\dd\sigma\dd q
 &=4\int_{\R^3}\int_{\R^3}
 \frac{H(a,q-a)}{|q|}\,
 \delta_0(|a|^2-a\cdot q)\,\dd a\dd q\\
 &=4\int_{\R^3}\int_{\R^3}
 \frac{H(a,c)}{|a+c|}\,
 \delta_0(a\cdot c)\,\dd c\dd a\\
 &=4\int_{\R^3}\frac{1}{|a|}
 \int_{a^\perp}\frac{H(a,c)}{r_{a,c}}
 \,\dd\mathcal H^2(c)\dd a,
\end{align*}
where $\delta_0$ denotes the Dirac mass at zero and we used $c=q-a$ in
the second line. In \eqref{eq:gain-intro}, set $q=v-v_*$, so that
$\dd v_*=\dd q$. Since $\sigma=(c-a)/r_{a,c}$ and $q=a+c$, taking
\[
 H(a,c)=\mathbf{b}\left(
 \frac{|c|^2-|a|^2}{|a|^2+|c|^2}\right)F(v-a)G(v-c)
\]
proves \eqref{eq:carleman-gain}. The change $\sigma\mapsto-\sigma$ exchanges
$a,c$, so the measure
$|a|^{-1}r_{a,c}^{-1}\dd a\dd\mathcal H^2(c)$ is symmetric in $a,c$.
Finally,
\[
 D_qa=\tfrac12(I-\sigma\otimes\widehat q),\qquad
 D_qc=\tfrac12(I+\sigma\otimes\widehat q),
\]
which gives \eqref{eq:collision-map-jacobians}. Injectivity follows from
the inverse formulas
\[
 q=2a-\frac{|a|^2}{a\cdot\sigma}\sigma
 \quad(a\cdot\sigma<0),\qquad
 q=2c-\frac{|c|^2}{c\cdot\sigma}\sigma
 \quad(c\cdot\sigma>0).
\]
\end{proof}

\begin{lemma}[Carleman integral estimate]
\label{lem:resonance-integral}
For every $(y,v)\ne(0,0)$,
\begin{equation}
 \mathfrak I(y,v):=
 \int_0^\infty\int_{\R^3}\int_{a^\perp}
 \frac{\Phi_m(y+sa,v-a)\Phi_m(y+sc,v-c)}{|a|r_{a,c}}
 \,\dd\mathcal H^2(c)\dd a\dd s
 \le C_m\Phi_m(y,v).
 \label{eq:resonance-integral}
\end{equation}
\end{lemma}
\begin{proof}
Write $L:=\varrho_m(y,v)>0$. For $a\ne0$, the plane integral used below is
\begin{equation}
 \int_{a^\perp}\frac{\dd\mathcal H^2(c)}
 {r_{a,c}(L^{2m}+r_{a,c}^{2m})^{(m+2)/(2m)}}
 =2\pi\int_{|a|}^\infty
 \frac{\dd r}{(L^{2m}+r^{2m})^{(m+2)/(2m)}}
 \le\frac{C_m}{(L+|a|)^{m+1}}.
 \label{eq:carleman-plane-bound}
\end{equation}
Indeed, use polar coordinates on $a^\perp$ and then
$r=(|a|^2+|c|^2)^{1/2}$.

Set $A:=|y+sa|^2+|v-a|^{2m}$ and
$B:=|y+sc|^2+|v-c|^{2m}$. Orthogonality gives
\[
 |y+sa|^2+|y+sc|^2=|y|^2+|y+s(a+c)|^2\ge|y|^2,
\]
while
$|v-a|^2+|v-c|^2=|v|^2+|v-a-c|^2$
is at least $|v|^2$ and at least $r_{a,c}^2/2$.
Since $m>1$, convexity also gives
\[
 |v-a|^{2m}+|v-c|^{2m}
 \geq 2^{1-m}\bigl(|v-a|^2+|v-c|^2\bigr)^m
 \geq c_m\bigl(|v|^{2m}+r_{a,c}^{2m}\bigr).
\]
Consequently,
\[
 A+B\ge C_m^{-1}(L^{2m}+r_{a,c}^{2m}),
\]
and hence
\[
 A^{-\frac{m+2}{2m}}B^{-\frac{m+2}{2m}}
 \le\frac{C_m
 (A^{-\frac{m+2}{2m}}+B^{-\frac{m+2}{2m}})}
 {(L^{2m}+r_{a,c}^{2m})^{(m+2)/(2m)}}.
\]
By symmetry in $a,c$ and \eqref{eq:carleman-plane-bound},
\begin{equation}
 \mathfrak I(y,v)
 \le C_m\int_{\R^3}\frac{1}{|a|(L+|a|)^{m+1}}
 \int_0^\infty\Phi_m(y+sa,v-a)\,\dd s\dd a.
 \label{eq:carleman-reduction}
\end{equation}
For $a\ne0$, let
$\Pi_a y:=y-(y\cdot a)a/|a|^2$.
Translate time so that the point on the line closest to the origin
corresponds to time zero. Extending the integral to $\R$ gives
\[
 \int_0^\infty\Phi_m(y+sa,v-a)\,\dd s
 \le\frac{C_m}{|a|}
 \bigl(|\Pi_a y|^2+|v-a|^{2m}\bigr)^{-1/m}.
\]
It remains to bound
\begin{equation}
 \int_{\R^3}
 \frac{\bigl(|\Pi_a y|^2+|v-a|^{2m}\bigr)^{-1/m}}
 {|a|^2(L+|a|)^{m+1}}\,\dd a
 \le C_mL^{-(m+2)}.
 \label{eq:carleman-final-integral}
\end{equation}
Since $L^{2m}=|y|^2+|v|^{2m}$, either
$|y|\ge L^m/\sqrt2$ or $|v|\ge2^{-1/(2m)}L$.
In the first case, write $a=r\omega$ and discard
the velocity term. The left-hand side of
\eqref{eq:carleman-final-integral} is at most
\[
 C_m|y|^{-2/m}\int_0^\infty\frac{\dd r}{(L+r)^{m+1}}
 \int_{\Sph^2}\bigl(1-(\widehat y\cdot\omega)^2\bigr)^{-1/m}
 \,\dd\omega
 \le C_m|y|^{-2/m}L^{-m}\le C_mL^{-(m+2)}.
\]
The angular integral is finite precisely because $m>1$. In the second
case, discard the spatial term and use
$(L+|a|)^{-(m+1)}\le L^{-(m+1)}$. Then
\[
 \int_{\R^3}\frac{\dd a}{|a|^2(L+|a|)^{m+1}|v-a|^2}
 \le C_mL^{-(m+1)}|v|^{-1}\le C_mL^{-(m+2)}.
\]
Here $\int_{\R^3}|a|^{-2}|v-a|^{-2}\,\dd a=C|v|^{-1}$ follows
by scaling and rotation; the singularities at $0,v$ and the tail are
integrable. This proves \eqref{eq:carleman-final-integral} and the lemma.
\end{proof}

\begin{proof}[Proof of Proposition~\ref{prop:critical-gain}]
At the spacetime point $(s,y+sv)$, the free spatial feet of the incoming
velocities $v-a$ and $v-c$ are $y+sa$ and $y+sc$.
Thus, for almost every choice of the variables,
\[
 \begin{aligned}
 F(s,y+sv,v-a)
 &=F^\#(s,y+sa,v-a)
 \leq \|F\|_{Y_m}\Phi_m(y+sa,v-a),\\
 G(s,y+sv,v-c)
 &=G^\#(s,y+sc,v-c)
 \leq \|G\|_{Y_m}\Phi_m(y+sc,v-c).
 \end{aligned}
\]
Using \eqref{eq:carleman-gain} and
$0\le k_{\mathbf{b}}(a,c)\le\|\mathbf{b}\|_\infty/r_{a,c}$, for
almost every $(t,y,v)$ with $(y,v)\ne(0,0)$ we obtain
\begin{align*}
 \mathcal B(F,G)^\#(t,y,v)
 &\le4\|\mathbf{b}\|_\infty\|F\|_{Y_m}\|G\|_{Y_m}\\[-2mm]
 &\quad\times\int_0^t\int_{\R^3}\int_{a^\perp}
 \frac{\Phi_m(y+sa,v-a)\Phi_m(y+sc,v-c)}{|a|r_{a,c}}
 \,\dd\mathcal H^2(c)\dd a\dd s\\
 &\le4\|\mathbf{b}\|_\infty\|F\|_{Y_m}\|G\|_{Y_m}
 \mathfrak I(y,v)\\
 &\le C_m\|\mathbf{b}\|_\infty
 \|F\|_{Y_m}\|G\|_{Y_m}\Phi_m(y,v).
\end{align*}
The point $(y,v)=(0,0)$ does not affect the essential supremum.
Multiplying by $\varrho_m(y,v)^{m+2}$ and taking the essential supremum
proves \eqref{eq:critical-gain}.
\end{proof}
The gain estimate and the construction below use only
$0\le\mathbf{b}\in L^\infty([-1,1])$. The condition
$\Lambda_{\mathbf{b}}<\infty$ is used later in the angular integration
identities and the localization estimates.
\subsection{Solving the equation without the loss term}
We now apply the gain estimate in a contraction argument. This gives
the solution used to start the upper sequence.
Let $f^{\mathrm{fr}}(t,x,v)=f_0(x-tv,v)$.  By
\eqref{eq:trace-weight-bound},
\begin{equation}
 \|f^{\mathrm{fr}}\|_{Y_m}\leq C_{m,\phi}\eps.
 \label{eq:free-Y-bound}
\end{equation}
Consider the map
\[
 \mathcal T(U):=f^{\mathrm{fr}}+\mathcal B(U,U).
\]
To construct a global mild solution of the equation without the loss term
\[
 (\partial_t+v\cdot\nabla_x)U=\Qp(U,U),
 \qquad U(0,x,v)=f_0(x,v),
\]
it suffices to find a nonnegative fixed point $U=\mathcal T(U)$ in $Y_m$.
Indeed, by \eqref{eq:B-definition}, this fixed-point identity is the
characteristic formula \eqref{eq:gain-only-mild}. We obtain the fixed
point by the contraction mapping theorem.
\begin{proposition}[Global solution of the equation without the loss term]
\label{prop:gain-only}
For $\eps$ sufficiently small, $\mathcal T$ has a unique nonnegative fixed
point in the ball
\[
 \{U:\|U\|_{Y_m}\leq2C_{m,\phi}\eps\}.
\]
It satisfies
\begin{equation}
 U^\#(t,y,v)=f_0(y,v)
 +\int_0^t\Qp(U,U)(s,y+sv,v)\,\dd s
 \label{eq:gain-only-mild}
\end{equation}
and
\begin{equation}
 0\leq U^\#(t,y,v)\leq A_{0,m}\eps
 \varrho_m(y,v)^{-(m+2)}.
 \label{eq:gain-envelope}
\end{equation}
Moreover, $U$ is exactly invariant under \eqref{eq:kinetic-scaling-intro}.
\end{proposition}
\begin{proof}
Proposition~\ref{prop:critical-gain} gives
\[
 \|\mathcal T(U)\|_{Y_m}
 \leq C_{m,\phi}\eps+C_{m,\mathbf{b}}\|U\|_{Y_m}^2
\]
and the corresponding Lipschitz estimate
\[
 \|\mathcal T(U)-\mathcal T(V)\|_{Y_m}
 \leq C_{m,\mathbf{b}}(\|U\|_{Y_m}+\|V\|_{Y_m})
 \|U-V\|_{Y_m}.
\]
For small $\eps$, this map is a contraction on the displayed ball and
preserves nonnegativity. To choose a pointwise representative of its
fixed point, consider the increasing Picard iterates
\[
 U^{(0)}=f^{\mathrm{fr}},\qquad
 U^{(j+1)}=f^{\mathrm{fr}}+\mathcal B(U^{(j)},U^{(j)}).
\]
Applying the pointwise estimate in the proof of
Proposition~\ref{prop:critical-gain} inductively, and using
the smallness of $\eps$, we obtain
\[
 0\le U^{(j)\#}(t,y,v)\le 2C_{m,\phi}\eps\Phi_m(y,v)
\]
for every $j\ge0$, every $t>0$, and every $(y,v)\ne(0,0)$.
Indeed, the bound holds for $U^{(0)}$ by \eqref{eq:trace-weight-bound}.
If it holds for $U^{(j)}$, then
\[
 U^{(j+1)\#}(t,y,v)
 \le \bigl[C_{m,\phi}\eps
 +C_{m,\mathbf{b}}(2C_{m,\phi}\eps)^2\bigr]\Phi_m(y,v)
 \le 2C_{m,\phi}\eps\Phi_m(y,v),
\]
provided $4C_{m,\mathbf{b}}C_{m,\phi}\eps\le1$.
By monotone convergence, the pointwise limit satisfies
\eqref{eq:gain-only-mild} and coincides with the fixed point $U$.
This formula gives local absolute continuity in time along every
nonzero free characteristic. Each Picard iterate satisfies
\eqref{eq:kinetic-scaling-intro}, so the pointwise limit does too,
for all $r>0$. We use this representative throughout and set
$U^\#(t,0,0)=0$; by Lemma~\ref{lem:cutoff-collision-coordinates},
this choice does not affect the collision integrals.
\end{proof}
We next estimate the velocity density, which determines the collision
frequency.
\begin{lemma}[Velocity density of the anisotropic envelope]
\label{lem:density-envelope}
For $t>0$ and $x\in\R^3$,
\begin{equation}
 \int_{\R^3}(|x-tv|^2+|v|^{2m})^{-\frac{m+2}{2m}}\,\dd v
 \leq C_mt^{-1}
 \bracket{\frac{x}{t^{m/(m-1)}}}^{-(m-1)/m}.
 \label{eq:weight-density-integral}
\end{equation}
Consequently,
\begin{equation}
 \rho_U(t,x)\leq A_{\rho,m}\eps t^{-1}
 \bracket{\frac{x}{t^{m/(m-1)}}}^{-(m-1)/m}.
 \label{eq:density-envelope}
\end{equation}
\end{lemma}
\begin{proof}
Set $v=t^{1/(m-1)}w$ and $X=x/t^{m/(m-1)}$. The left-hand side of
\eqref{eq:weight-density-integral} equals
\[
 t^{-1}J_m(X),
 \qquad
 J_m(X):=\int_{\R^3}
 (|X-w|^2+|w|^{2m})^{-\frac{m+2}{2m}}\,\dd w.
\]
Let $R_m:=\max\{8,4^{m/(m-1)}\}$. For $|X|\le R_m$, we have
\[
\begin{aligned}
 J_m(X)
 &\le \int_{|w|\le2R_m}|X-w|^{-(m+2)/m}\,\dd w
       +\int_{|w|>2R_m}|w|^{-(m+2)}\,\dd w\\
 &\le \int_{|z|\le3R_m}|z|^{-(m+2)/m}\,\dd z
       +4\pi\int_{2R_m}^{\infty}r^{-m}\,\dd r
 \le C_m.
\end{aligned}
\]
Both radial integrals are finite because $m>1$.
For $R=|X|\ge R_m$, we use
\[
 |w|\le2R^{1/m}
 \quad\Longrightarrow\quad
 |X-w|\ge R-2R^{1/m}\ge R/2
\]
to obtain
\[
\begin{aligned}
 J_m(X)
 &\le \int_{|w|\le2R^{1/m}}|X-w|^{-(m+2)/m}\,\dd w
       +\int_{|w|>2R^{1/m}}|w|^{-(m+2)}\,\dd w\\
 &\le (R/2)^{-(m+2)/m}\frac{4\pi}{3}(2R^{1/m})^3
       +4\pi\int_{2R^{1/m}}^{\infty}r^{-m}\,\dd r\\
 &\le C_mR^{-(m-1)/m}.
\end{aligned}
\]
Thus $J_m(X)\le C_m\bracket{X}^{-(m-1)/m}$, which proves
\eqref{eq:weight-density-integral}. The bound
\eqref{eq:density-envelope} follows from \eqref{eq:gain-envelope}.
\end{proof}

\begin{lemma}[Density along a nonzero-velocity line]
\label{lem:line-bound-main}
Let $a>0$, $z\in\R^3$, $|q|\geq a$, and $t>0$. Then
\begin{equation}
 \int_0^t \rho_U(s,z+sq)\,\dd s
 \leq
 C_mA_{\rho,m}\eps\,
 a^{-(m-1)/m}t^{1/m}.
 \label{eq:line-bound-main}
\end{equation}
\end{lemma}
\begin{proof}
From \eqref{eq:density-envelope},
\[
 \rho_U(s,x)\leq A_{\rho,m}\eps |x|^{-(m-1)/m}.
\]
With $s_0=-z\cdot q/|q|^2$, orthogonal projection gives
$|z+sq|\geq|q||s-s_0|$. Since
$|s-s_0|^{-(m-1)/m}$ is locally integrable,
\[
 \int_0^t\rho_U(s,z+sq)\,\dd s
 \leq A_{\rho,m}\eps a^{-(m-1)/m}
 \int_0^t|s-s_0|^{-(m-1)/m}\,\dd s,
\]
which proves the claim.
\end{proof}

\section{Kaniel--Shinbrot brackets}
\label{sec:KS}
We use the solution $U$ of the equation without the loss term to start
the Kaniel--Shinbrot iteration. This gives nonnegative lower approximations
$f_n^{\mathrm{lo}}$ that increase and upper approximations
$f_n^{\mathrm{up}}$ that decrease. We will prove that their limits agree
and solve the full Boltzmann equation.

For a nonnegative function $g$ with finite velocity density, recall
$\nu_g=\beta_{\mathbf{b}}\rho_g$.  Starting with
\[
 f_0^{\mathrm{lo}}=0,
 \qquad f_0^{\mathrm{up}}=U,
\]
define $f_{n+1}^{\mathrm{lo}},f_{n+1}^{\mathrm{up}}$ successively by the
characteristic formulas for
\begin{align}
 (\partial_t+v\cdot\nabla_x)f_{n+1}^{\mathrm{lo}}
 +\nu_{f_n^{\mathrm{up}}}f_{n+1}^{\mathrm{lo}}
 &=\Qp(f_n^{\mathrm{lo}},f_n^{\mathrm{lo}}),
 \label{eq:lower-iterate}\\
 (\partial_t+v\cdot\nabla_x)f_{n+1}^{\mathrm{up}}
 +\nu_{f_n^{\mathrm{lo}}}f_{n+1}^{\mathrm{up}}
 &=\Qp(f_n^{\mathrm{up}},f_n^{\mathrm{up}}),
 \label{eq:upper-iterate}
\end{align}
with initial trace $f_0$.
\begin{proposition}[Monotone bracket]
\label{prop:KS-bracket}
Assume that $\eps>0$ is small enough for
Proposition~\ref{prop:gain-only} to apply, and let
$U=U_\eps$ be the corresponding solution of the equation without the
loss term. Thus
\[
 \|U_\eps\|_{Y_m}\leq 2C_{m,\phi}\eps.
\]
Let $(f_n^{\mathrm{lo}},f_n^{\mathrm{up}})$ be the Kaniel--Shinbrot
iterates defined by
\eqref{eq:lower-iterate}--\eqref{eq:upper-iterate}.  Then
\begin{equation}
 0=f_0^{\mathrm{lo}}\leq f_1^{\mathrm{lo}}\leq\cdots
 \leq f_n^{\mathrm{lo}}\leq f_n^{\mathrm{up}}
 \leq\cdots\leq f_1^{\mathrm{up}}\leq f_0^{\mathrm{up}}=U.
 \label{eq:monotone-bracket}
\end{equation}
Consequently, the pointwise limits
\[
 f^{\mathrm{lo}}:=\lim_{n\to\infty}f_n^{\mathrm{lo}},
 \qquad
 f^{\mathrm{up}}:=\lim_{n\to\infty}f_n^{\mathrm{up}}
\]
exist and satisfy
\begin{equation}
 \begin{split}
  (\partial_t+v\cdot\nabla_x)f^{\mathrm{lo}}
  +\nu_{f^{\mathrm{up}}}f^{\mathrm{lo}}
  &=\Qp(f^{\mathrm{lo}},f^{\mathrm{lo}}),\\
  (\partial_t+v\cdot\nabla_x)f^{\mathrm{up}}
  +\nu_{f^{\mathrm{lo}}}f^{\mathrm{up}}
  &=\Qp(f^{\mathrm{up}},f^{\mathrm{up}}),
 \end{split}
 \qquad
 0\leq f^{\mathrm{lo}}\leq f^{\mathrm{up}}\leq U,
 \label{eq:KS-coupled-limits-main}
\end{equation}
in the characteristic mild sense. Moreover, both $f^{\mathrm{lo}}$ and
$f^{\mathrm{up}}$ are exactly invariant under the kinetic scaling
\eqref{eq:kinetic-scaling-intro}. Here $f_0$, $U$,
$f_n^{\mathrm{lo}}$, $f_n^{\mathrm{up}}$, $f^{\mathrm{lo}}$, and
$f^{\mathrm{up}}$ all depend on $\eps$; this dependence is suppressed in
the notation.
\end{proposition}
\begin{proof}
Fix such an $\eps$ and recall
$f^{\mathrm{fr}}(t,x,v)=f_0(x-tv,v)$, so that
$(f^{\mathrm{fr}})^\#=f_0$.
We use the following comparison property. Let $a,h\ge0$, with $a$ having
finite velocity density almost everywhere. Assume that, for almost every
$(y,v)$ and every $0<T<\infty$,
\[
 \int_0^T\nu_a(s,y+sv)\,\dd s<\infty,
 \qquad
 \int_0^T h(s,y+sv,v)\,\dd s<\infty.
\]
Denote by $\mathcal L[a,h]$ the unique solution along almost every
characteristic of
\[
 (\partial_t+v\cdot\nabla_x)F+\nu_aF=h,
 \qquad F|_{t=0}=f_0.
\]
Its characteristic formula is
\[
 \begin{split}
 \mathcal L[a,h]^\#(t,y,v)
 ={}&
 \exp\left(-\int_0^t\nu_a(s,y+sv)\,\dd s\right)f_0(y,v)\\
 &+\int_0^t
 \exp\left(-\int_s^t\nu_a(\tau,y+\tau v)\,\dd\tau\right)
 h(s,y+sv,v)\,\dd s.
 \end{split}
\]
It follows directly from this formula that
\begin{equation}
 a_1\leq a_2,\qquad h_1\leq h_2
 \quad\Longrightarrow\quad
 \mathcal L[a_2,h_1]\leq\mathcal L[a_1,h_2].
 \label{eq:attenuated-comparison}
\end{equation}
In words, increasing the collision frequency decreases the solution, whereas
increasing the source increases the solution.
We verify these integrability assumptions before applying the comparison.
By
Lemma~\ref{lem:line-bound-main} with $z=y$, $q=v$, $a=|v|$, 
\[ 
\int_0^T\nu_U(s,y+sv)\,\dd s<\infty,
\] 
whenever $v\neq0$. If $v=0$, then $y\neq0$, and the same conclusion
follows directly from \eqref{eq:density-envelope}.
The identity
\eqref{eq:gain-only-mild} also gives
\[
 \int_0^T\Qp(U,U)(s,y+sv,v)\,\dd s
 =U^\#(T,y,v)-f_0(y,v)<\infty.
\]
Thus the assumptions hold whenever $0\le a\le U$ and
$0\le h\le\Qp(U,U)$.
The iteration can therefore be constructed inductively as
\[
 f_{n+1}^{\mathrm{lo}}
 =\mathcal L[f_n^{\mathrm{up}},
             \Qp(f_n^{\mathrm{lo}},f_n^{\mathrm{lo}})],
 \qquad
 f_{n+1}^{\mathrm{up}}
 =\mathcal L[f_n^{\mathrm{lo}},
             \Qp(f_n^{\mathrm{up}},f_n^{\mathrm{up}})].
\]
For the first upper iterate, since $f_0^{\mathrm{lo}}=0$ and hence
$\nu_{f_0^{\mathrm{lo}}}=0$, the fixed-point identity
\eqref{eq:gain-only-mild} gives
\[
 f_1^{\mathrm{up}}
 =f^{\mathrm{fr}}+\mathcal B(U,U)
 =U
 =f_0^{\mathrm{up}}.
\]
On the other hand,
\[
 f_1^{\mathrm{lo}}=\mathcal L[U,0],
\]
so that
\[
 0=f_0^{\mathrm{lo}}\leq f_1^{\mathrm{lo}}
 \leq f^{\mathrm{fr}}\leq U=f_1^{\mathrm{up}}.
\]
This proves the first step of the bracket.
Suppose now that, for some $n\geq1$,
\[
 f_{n-1}^{\mathrm{lo}}\leq f_n^{\mathrm{lo}}
 \leq f_n^{\mathrm{up}}\leq f_{n-1}^{\mathrm{up}}.
\]
Since $\Qp$ is order preserving in each argument,
\[
 \Qp(f_{n-1}^{\mathrm{lo}},f_{n-1}^{\mathrm{lo}})
 \leq\Qp(f_n^{\mathrm{lo}},f_n^{\mathrm{lo}})
 \leq\Qp(f_n^{\mathrm{up}},f_n^{\mathrm{up}})
 \leq\Qp(f_{n-1}^{\mathrm{up}},f_{n-1}^{\mathrm{up}}).
\]
Combining these inequalities with
\eqref{eq:attenuated-comparison} gives
\[
 f_n^{\mathrm{lo}}\leq f_{n+1}^{\mathrm{lo}},
 \qquad
 f_{n+1}^{\mathrm{lo}}\leq f_{n+1}^{\mathrm{up}},
 \qquad
 f_{n+1}^{\mathrm{up}}\leq f_n^{\mathrm{up}}.
\]
Induction therefore proves \eqref{eq:monotone-bracket}.
The monotone sequences are bounded by $U$, so their pointwise limits
$f^{\mathrm{lo}}$ and $f^{\mathrm{up}}$ exist and satisfy
\[
 0\leq f^{\mathrm{lo}}\leq f^{\mathrm{up}}\leq U.
\]
The collision frequencies $\nu_{f_n^{\mathrm{lo}}}$ and
$\nu_{f_n^{\mathrm{up}}}$ are bounded by $\nu_U$, which is integrable
along each nonzero characteristic, as shown above.
Moreover,
\[
 \Qp(f_n^{\mathrm{lo}},f_n^{\mathrm{lo}})
 \longrightarrow\Qp(f^{\mathrm{lo}},f^{\mathrm{lo}}),
 \qquad
 \Qp(f_n^{\mathrm{up}},f_n^{\mathrm{up}})
 \longrightarrow\Qp(f^{\mathrm{up}},f^{\mathrm{up}}),
\]
where the second convergence is dominated by $\Qp(U,U)$, whose time
integral along each nonzero characteristic is finite.
Monotone convergence for the lower source and dominated convergence for the
upper source and both exponential loss factors now permit passage to the limit in
the characteristic formulas.  This yields
\eqref{eq:KS-coupled-limits-main}.
It remains to verify the scaling.  Let
\[
 (\mathcal D_r^{(m)}h)(t,x,v)
 :=r^{m+2}h(r^{m-1}t,r^mx,rv),
 \qquad r>0.
\]
The homogeneity of $f_0$ implies
$\mathcal D_r^{(m)}f^{\mathrm{fr}}=f^{\mathrm{fr}}$, while
Proposition~\ref{prop:gain-only} gives $\mathcal D_r^{(m)}U=U$. Thus the initial
approximations $f_0^{\mathrm{lo}}=0$ and $f_0^{\mathrm{up}}=U$ are invariant
under this scaling. Assume that $f_n^{\mathrm{lo}}$ and $f_n^{\mathrm{up}}$ are
scale invariant. The scaling of the collision operator shows that
$\mathcal D_r^{(m)}f_{n+1}^{\mathrm{lo}}$ and
$f_{n+1}^{\mathrm{lo}}$ solve the same linear equation, with the same
collision frequency, source, and initial trace.
Uniqueness of the characteristic formula gives
\[
 \mathcal D_r^{(m)}f_{n+1}^{\mathrm{lo}}=f_{n+1}^{\mathrm{lo}}.
\]
The same argument gives
$\mathcal D_r^{(m)}f_{n+1}^{\mathrm{up}}=f_{n+1}^{\mathrm{up}}$.
Induction and
passage to the pointwise limits show that
\[
 \mathcal D_r^{(m)}f^{\mathrm{lo}}=f^{\mathrm{lo}},
 \qquad
 \mathcal D_r^{(m)}f^{\mathrm{up}}=f^{\mathrm{up}}
\]
for every $r>0$.
The smallness of $\eps$ is used here only to construct and control the
global solution $U_\eps$ of the equation without the loss term in
Proposition~\ref{prop:gain-only}; once $U_\eps$ is available, the monotone
bracket argument itself requires no additional smallness.
\end{proof}
\paragraph{Characteristic representatives.}
Away from $(y,v)=(0,0)$, we define the iterates by their characteristic
formulas and the limits by pointwise convergence. The integrable bounds
for the collision frequency and gain term allow us to pass to the limit
in these formulas. The limits are therefore absolutely continuous in
time along each such characteristic. They also satisfy the scaling for
all $r>0$ simultaneously, since every iterate does.
We use these representatives when evaluating at a fixed time. The
equation is understood almost everywhere, and the values on
$(y,v)=(0,0)$ do not affect the essential suprema below. With $A_h$ as in
\eqref{eq:mild-definition}, the limit formulas are, almost everywhere,
\begin{align}
 (f^{\mathrm{lo}})^\#(t,y,v)
 &=e^{-A_{f^{\mathrm{up}}}(0,t;y,v)}f_0(y,v)
 +\int_0^t e^{-A_{f^{\mathrm{up}}}(s,t;y,v)}
 \Qp(f^{\mathrm{lo}},f^{\mathrm{lo}})(s,y+sv,v)\,\dd s,
 \label{eq:lower-limit-formula}\\
 (f^{\mathrm{up}})^\#(t,y,v)
 &=e^{-A_{f^{\mathrm{lo}}}(0,t;y,v)}f_0(y,v)
 +\int_0^t e^{-A_{f^{\mathrm{lo}}}(s,t;y,v)}
 \Qp(f^{\mathrm{up}},f^{\mathrm{up}})(s,y+sv,v)\,\dd s.
 \label{eq:upper-limit-formula}
\end{align}
\section{Short-time estimates for the bracket difference}
\label{sec:gap-flatness}
We estimate the difference between the upper and lower limits constructed
in Section~\ref{sec:KS}. On a fixed annulus in phase space, this difference
tends to zero faster than any prescribed power of time as $t\downarrow0$.

Let
\[
 w:=f^{\mathrm{up}}-f^{\mathrm{lo}},
 \qquad
 \mathcal A_m:=\{(y,v):1/2\leq\varrho_m(y,v)\leq2\}.
\]
More precisely, we prove that, for every integer $N\ge1$,
\begin{equation}
 \esssup_{(y,v)\in\mathcal A_m}w^\#(t,y,v)\leq C_{m,N}t^N
 \qquad(0<t\leq t_{m,N}).
 \label{eq:target-gap-flatness}
\end{equation}
\subsection{Line and high-velocity estimates}
We first integrate out one incoming velocity. Energy conservation then
allows us to estimate the gain term when the output velocity is large.

\begin{lemma}[Angular integration identities]
\label{lem:angular-integration}
Under \eqref{eq:strong-cutoff-intro}, set
\[
 K_\pm:=2\pi\int_{-1}^1\mathbf{b}(z)
 \left(\frac{1\pm z}{2}\right)^{-3/2}\,\dd z.
\]
Then $K_++K_-=2^{5/2}\pi\Lambda_{\mathbf{b}}<\infty$.
For every nonnegative $h\in L^1(\R^3)$ and every $v\in\R^3$,
\begin{align}
 \int_{\R^3}\int_{\Sph^2}
 \mathbf{b}(\widehat q\cdot\sigma)h(v-a)\,\dd\sigma\dd q
 &=K_-\int_{\R^3}h(w)\,\dd w,
 \label{eq:angular-integration-minus}\\
 \int_{\R^3}\int_{\Sph^2}
 \mathbf{b}(\widehat q\cdot\sigma)h(v-c)\,\dd\sigma\dd q
 &=K_+\int_{\R^3}h(w)\,\dd w,
 \label{eq:angular-integration-plus}
\end{align}
where $a=(q-|q|\sigma)/2$ and $c=(q+|q|\sigma)/2$.
\end{lemma}
\begin{proof}
Write $q=r\omega$ and
$\sigma=z\omega+\sqrt{1-z^2}\eta$, where
$\eta\in\omega^\perp\cap\Sph^2$ and $-1<z<1$.
Then $\dd\sigma=\dd z\dd\mathcal H^1(\eta)$. Set
\[
 d_\pm(z):=\sqrt{\frac{1\pm z}{2}}.
\]
At fixed $z$, the angular measure
$\dd\omega\dd\mathcal H^1(\eta)$ is invariant under simultaneous
rotations and has total mass $8\pi^2$. Since
$|a|=rd_-(z)$ and $|c|=rd_+(z)$, the directions of both displacements
are uniformly distributed on $\Sph^2$. Hence the left-hand side of
\eqref{eq:angular-integration-minus} equals
\[
 2\pi\int_{-1}^1\mathbf{b}(z)\int_0^\infty r^2
 \int_{\Sph^2}h(v-rd_-(z)n)\,\dd n\dd r\dd z.
\]
The substitution $\rho=rd_-(z)$ gives
$K_-\int_{\R^3}h(w)\,\dd w$. Replacing $d_-$ by $d_+$ proves
\eqref{eq:angular-integration-plus}. All integrands are nonnegative,
so Tonelli's theorem justifies the changes in integration order.
\end{proof}

For $a>0$, set
\[
 \mathcal S_aF(x):=\esssup_{|q|\geq a}F(x,q).
\]
\begin{lemma}[High-velocity splitting]
\label{lem:high-velocity-main}
For nonnegative $F,G$, set $C_{\mathbf{b}}:=K_++K_-$. Then
\begin{equation}
 \mathbf1_{\{|v|\geq a\}}\Qp(F,G)(x,v)
 \leq C_{\mathbf{b}}\bigl(
 \rho_G(x)\mathcal S_{a/\sqrt2}F(x)
 +\rho_F(x)\mathcal S_{a/\sqrt2}G(x)\bigr).
 \label{eq:high-velocity-main}
\end{equation}
In particular, before replacing the actual threshold by $a/\sqrt2$, the
same argument gives, for every nonnegative $H$,
\begin{equation}
 \Qp(H,H)(x,p)
 \leq C_{\mathbf{b}}\rho_H(x)
 \esssup_{|q|\geq |p|/\sqrt2}H(x,q).
 \label{eq:adaptive-high-velocity}
\end{equation}
\end{lemma}
\begin{proof}
Energy conservation implies
$|v|^2\leq|v'|^2+|v_*'|^2$. Thus, on $\{|v|\geq a\}$, one incoming
velocity has modulus at least $a/\sqrt2$, and
\[
 \begin{aligned}
 F(x,v')G(x,v_*')
\leq \mathcal S_{a/\sqrt2}F(x)G(x,v_*')+F(x,v')\mathcal S_{a/\sqrt2}G(x).
 \end{aligned}
\]
Integrating this inequality and applying
Lemma~\ref{lem:angular-integration}, with $q=v-v_*$, gives
\[
 K_+\rho_G(x)\mathcal S_{a/\sqrt2}F(x)
 +K_-\rho_F(x)\mathcal S_{a/\sqrt2}G(x).
\]
Since $K_\pm\leq C_{\mathbf{b}}$, this proves
\eqref{eq:high-velocity-main}.
If $F=G=H$ and the threshold $|v|/\sqrt2$ is retained instead of being
weakened to $a/\sqrt2$, the same calculation proves
\eqref{eq:adaptive-high-velocity}.
\end{proof}
\subsection{The inner region}
We now estimate the solution when the initial spatial position $y$ is
close to zero.
Repeated use of the high-velocity estimate gives arbitrarily high powers
of time in this region.
\begin{lemma}[Finite collision-tree flatness]
\label{lem:inner-flatness-main}
For every integer $N\geq1$, there exist
$\delta_{m,N},t_{m,N}>0$ and $C_{m,N}<\infty$ such that
\begin{equation}
\esssup_{\substack{(y,v)\in\mathcal A_m\\ |y|\leq\delta_{m,N}}}
U^\#(t,y,v)
\leq C_{m,N}\eps^{N+1}t^{N/m},
\qquad 0<t\leq t_{m,N}.
\label{eq:inner-flatness-main}
\end{equation}
The constants may also depend on the fixed kernel $\mathbf b$ and profile
$\phi$.
\end{lemma}

\begin{proof}
Fix $N\geq1$.  We iterate the high-velocity estimate $N$ times.
At each step the free term vanishes, while the time integral contributes
a factor of order $\eps t^{1/m}$.

We use the pointwise characteristic representative of $U$ fixed in
Proposition~\ref{prop:gain-only}, and set $U(0,X,p):=f_0(X,p)$. Since $\varrho_m(y,p)\geq1/2$ on $\mathcal A_m$, we may choose
$\delta_*,a_0>0$, depending only on $m$, such that
\[
 (y,p)\in\mathcal A_m,\qquad |y|\leq\delta_*
\quad\Longrightarrow\quad |p|\geq a_0.
\]
Set
$a_j:=2^{-j/2}a_0$ for $0\leq j\leq N$.

\smallskip
\noindent\emph{Step 1. The free-transport terms vanish.}
We keep track of the points at which the high-velocity estimate is
applied. Let $\mathcal R_0(t)$ consist of all triples $(s_0,X_0,p_0)$ such
that
\begin{equation}
0<s_0\leq t,\qquad
X_0=y_0+s_0p_0,\qquad
(y_0,p_0)\in\mathcal A_m,\qquad
|y_0|\leq\delta_{m,N},
\label{eq:reachable-root}
\end{equation}
where $\delta_{m,N}\leq\delta_*$ will be chosen below.

For $0\leq j<N$, define $\mathcal R_{j+1}(t)$ recursively as follows.
From any $(s_j,X_j,p_j)\in\mathcal R_j(t)$, choose
\begin{equation}
0\leq s_{j+1}\leq s_j,\qquad
X_{j+1}=X_j-(s_j-s_{j+1})p_j,\qquad
|p_{j+1}|\geq |p_j|/\sqrt2.
\label{eq:reachable-transition}
\end{equation}
The second relation follows the current characteristic back to an earlier
time. The third includes every velocity in the supremum in
\eqref{eq:adaptive-high-velocity}. We can therefore repeat that estimate
at all the points needed in the next step. In particular,
$|p_j|\geq a_j>0$ for $0\leq j\leq N$.
Hence all characteristics used below are nonzero, and the characteristic
formula for the chosen representative of $U$ applies to them.

For $(s_j,X_j,p_j)\in\mathcal R_j(t)$, set $Y_j:=X_j-s_jp_j$. The free term there is $f_0(Y_j,p_j)$.
Along any sequence of choices leading to this point,
\[
 Y_j-y_0=s_0p_0-\sum_{i=0}^{j-1}(s_i-s_{i+1})p_i-s_jp_j.
\]
Moreover, \eqref{eq:reachable-transition} gives
$|p_i|\leq2^{(j-i)/2}|p_j|$ for $0\leq i\leq j$, and
$\sum_{i=0}^{j-1}(s_i-s_{i+1})\leq t$.
Consequently, with $D_N:=2^{N/2+1}+1$,
\begin{equation}
 \begin{aligned}
 |Y_j-y_0|
 &\leq s_0|p_0|+\sum_{i=0}^{j-1}(s_i-s_{i+1})|p_i|+s_j|p_j|\\
 &\leq (2^{j/2+1}+1)t|p_j|\leq D_Nt|p_j|,
 \qquad 0\leq j\leq N.
 \end{aligned}
 \label{eq:reachable-foot-bound}
\end{equation}
Choose $\delta_{m,N}\leq\delta_*$ and $t_{m,N}>0$ so that
\[
 R_0^m\delta_{m,N}\leq\frac14a_N^m,
 \qquad R_0^mD_Nt_{m,N}\leq\frac14a_N^{m-1}.
\]
If $t\leq t_{m,N}$ and $f_0(Y_j,p_j)\ne0$, the support condition
\eqref{eq:conic-support-main} would imply
\[
 \begin{aligned}
 |p_j|^m
 &\leq R_0^m|Y_j|
 \leq R_0^m\bigl(\delta_{m,N}+D_Nt|p_j|\bigr)\\
 &\leq\frac14a_N^m+\frac14a_N^{m-1}|p_j|
\leq\frac12|p_j|^m.
 \end{aligned}
\]
The last inequality uses $|p_j|\geq a_j\geq a_N$ and $m>1$.
This is impossible, so $f_0(Y_j,p_j)=0$ for every point in
$\mathcal R_j(t)$, $0\leq j\leq N$.

\smallskip
\noindent\emph{Step 2. Iterating the high-velocity estimate.}
For $0<t\leq t_{m,N}$, define
\begin{equation}
H_j(t):=
\sup_{(s,X,p)\in\mathcal R_j(t)}U(s,X,p),
\qquad 0\leq j\leq N.
\label{eq:reachable-envelope}
\end{equation}
We use ordinary suprema of the pointwise representative fixed above.
After enlarging $A_{0,m}$ if necessary,
\eqref{eq:gain-envelope} and \eqref{eq:trace-weight-bound} give, for
$s\geq0$,
\[
 U(s,X,p)
 \leq A_{0,m}\eps\,
 \varrho_m(X-sp,p)^{-(m+2)}
 \leq A_{0,m}\eps |p|^{-(m+2)}.
\]
Thus the quantities $H_j(t)$ are finite, and in particular
\begin{equation}
H_N(t)\leq A_{0,m}\eps a_N^{-(m+2)}.
\label{eq:terminal-reachable-envelope}
\end{equation}

Fix $j<N$ and $(s,X,p)\in\mathcal R_j(t)$, and write $Z(\tau):=X-(s-\tau)p$. By the characteristic formula for $U$,
\[
 U(s,X,p)
 =
 \int_0^s\Qp(U,U)(\tau,Z(\tau),p)\,\dd\tau.
\]

Using \eqref{eq:adaptive-high-velocity}, we obtain

$$
 U(s,X,p)
 \leq
 C_{\mathbf b}\int_0^s
 \rho_U(\tau,Z(\tau))
 \esssup_{|q|\geq|p|/\sqrt2}
 U(\tau,Z(\tau),q)\,\dd\tau.
$$

For every $\tau\in[0,s]$ and every
$|q|\geq|p|/\sqrt2$, the point $
 (\tau,Z(\tau),q)
$ belongs to $\mathcal R_{j+1}(t)$.  Hence
\[
 U(s,X,p)
 \leq
 C_{\mathbf b}H_{j+1}(t)
 \int_0^s
 \rho_U\bigl(\tau,(X-sp)+\tau p\bigr)\,\dd\tau.
\]

Since $|p|\geq a_j$, Lemma~\ref{lem:line-bound-main} yields
\[
 \int_0^s
 \rho_U\bigl(\tau,(X-sp)+\tau p\bigr)\,\dd\tau
 \leq
 C_mA_{\rho,m}\eps
 a_j^{-(m-1)/m}t^{1/m}.
\]

Taking the supremum over $\mathcal R_j(t)$ therefore gives
\begin{equation}
H_j(t)
\leq
C_mC_{\mathbf b}A_{\rho,m}\eps
a_j^{-(m-1)/m}t^{1/m}
H_{j+1}(t),
\qquad 0\leq j<N.
\label{eq:reachable-envelope-recursion}
\end{equation}

Iterating \eqref{eq:reachable-envelope-recursion} from $j=0$ to $j=N-1$
and using \eqref{eq:terminal-reachable-envelope}, we obtain
\[
\begin{aligned}
 H_0(t)
\leq
 A_{0,m}\eps a_N^{-(m+2)}
 \prod_{j=0}^{N-1}
 \bigl(
 C_mC_{\mathbf b}A_{\rho,m}\eps
 a_j^{-(m-1)/m}t^{1/m}
 \bigr)\leq C_{m,N}\eps^{N+1}t^{N/m}.
\end{aligned}
\]

Finally, if $(y,v)\in\mathcal A_m$ and
$|y|\leq\delta_{m,N}$, then $(t,y+tv,v)\in\mathcal R_0(t)$. Therefore
\[
 U^\#(t,y,v)
 =
 U(t,y+tv,v)
 \leq H_0(t)
 \leq C_{m,N}\eps^{N+1}t^{N/m},
\]
which proves \eqref{eq:inner-flatness-main}.
\end{proof}

\subsection{Weighted collision and high-velocity estimates}
We next estimate the gain term in a weighted velocity norm and use the
inner-region estimate to control large velocities. These bounds will be
used in the region away from $y=0$.
For a velocity function $H(x,\cdot)$, write
\[
 \|H(x)\|_{\infty,m+2}
 :=\esssup_{v\in\R^3}\bracket{v}^{m+2}|H(x,v)|.
\]
\begin{lemma}[Local weighted collision bound]
\label{lem:local-weighted-collision}
For nonnegative $F,G$, under \eqref{eq:strong-cutoff-intro},
\begin{equation}
 \|\Qp(F,G)(x)\|_{\infty,m+2}
 \leq C_{m,\mathbf{b}}
 \bigl(\|F(x)\|_{\infty,m+2}\rho_G(x)
 +\rho_F(x)\|G(x)\|_{\infty,m+2}\bigr).
 \label{eq:local-weighted-gain}
\end{equation}
Consequently, for $U,H\ge0$, since
$\rho_H\leq C_m\|H\|_{\infty,m+2}$,
\begin{equation}
 \|\mathcal K_UH(x)\|_{\infty,m+2}
 \leq C_{m,\mathbf{b}}
 \|U(x)\|_{\infty,m+2}\|H(x)\|_{\infty,m+2},\,
 \mathcal K_UH:=\Qp(H,U)+\Qp(U,H)+\beta_{\mathbf{b}}\rho_HU.
 \label{eq:local-weighted-K}
\end{equation}
\end{lemma}
\begin{proof}
Energy conservation gives
\[
 \bracket{v}^{m+2}\leq C_m\bigl(\bracket{v'}^{m+2}+
 \bracket{v_*'}^{m+2}\bigr).
\]
For the first resulting term, take the weighted supremum of the factor at
$v'$ and apply \eqref{eq:angular-integration-plus} to the factor at
$v_*'$. This gives $C_mK_+\|F(x)\|_{\infty,m+2}\rho_G(x)$.
For the second term, take the weighted supremum at $v_*'$ and apply
\eqref{eq:angular-integration-minus}, giving
$C_mK_-\rho_F(x)\|G(x)\|_{\infty,m+2}$. This proves
\eqref{eq:local-weighted-gain}. For signed functions the same bound holds
with $\rho_{|F|}$ and $\rho_{|G|}$, by positivity of the gain kernel.
The last assertion follows from
$\int_{\R^3}\bracket{v}^{-(m+2)}\,\dd v<\infty$, since $m>1$.
\end{proof}
\begin{lemma}[Rescaled high-velocity flatness]
\label{lem:rescaled-high-flatness}
Fix $R_*,d>0$ and an integer $N\geq1$.  There are $C_{m,N},t_{m,N}>0$
such that
\begin{equation}
 \esssup_{\substack{0<s\leq t,\ |x|\leq R_*\\ |v|\geq d/t}}
 \bracket{v}^{m+2} U(s,x,v)
 \leq C_{m,N}\eps^{N+1}t^N,
 \qquad 0<t\leq t_{m,N}.
 \label{eq:rescaled-high-flatness}
\end{equation}
The constants may also depend on the fixed parameters $R_*,d,\mathbf{b},\phi$.
\end{lemma}
\begin{proof}
Put $R=|v|$, $r=R^{-1}$, and $y=x-sv$.  Exact self-similarity gives
\begin{equation}
 U(s,x,v)=U^\#(s,y,v)
 =r^{m+2}U^\#(r^{m-1}s,r^my,rv).
 \label{eq:high-flat-rescaling}
\end{equation}
Here $|rv|=1$, while, uniformly in the displayed region,
\[
 |r^my|\leq r^mR_*+sr^{m-1}\leq C_{m,R_*,d}t^m,
 \qquad r^{m-1}s\leq C_{m,d}t^m.
\]
Thus $(r^my,rv)$ belongs to the normalized annulus $\mathcal A_m$ and its
spatial component lies in the inner region of
Lemma~\ref{lem:inner-flatness-main}, once $t$ is small.  Applying that lemma
at time $r^{m-1}s$ gives
\[
 \begin{aligned}
 U(s,x,v)
 \leq C_{m,N}\eps^{N+1}R^{-(m+2)}(r^{m-1}s)^{N/m}
 \leq C_{m,N}\eps^{N+1}R^{-(m+2)}t^N.
 \end{aligned}
\]
Since $R\geq d/t\geq1$ after reducing $t_{m,N}$,
$\bracket{v}^{m+2}R^{-(m+2)}\leq C_m$, proving the claim. The estimate holds
pointwise for the chosen representative of $U$, by the pointwise version
of Lemma~\ref{lem:inner-flatness-main} proved above.
\end{proof}
\subsection{Local trace and stopped characteristics}
We first prove convergence to the initial trace away from the spatial
origin. We then estimate a function with zero initial trace by integrating
along characteristics until time zero or their first exit from a spatial
annulus.
Set
\begin{equation}
 B_m:=2^m+1.
\end{equation}
Thus $(y,v)\in\mathcal A_m$ implies $|y|\le2^m$ and $|v|\le2$,
so $|y+tv|\le B_m$ when $0<t\le1/2$.
\begin{lemma}[Exterior initial trace]
\label{lem:exterior-trace-direct}
Fix $0<\delta\le1$ and $R\ge1$. There are $t_{\delta,R}>0$ and
$C_{\delta,R}<\infty$ such that, for
$H\in\{U,f^{\mathrm{lo}},f^{\mathrm{up}}\}$,
\begin{equation}
 \esssup_{\substack{\delta\le|y|\le R\\|v|\le R}}
 |H^\#(t,y,v)-f_0(y,v)|\le C_{\delta,R}\eps^2t,
 \qquad 0<t\le t_{\delta,R}.
 \label{eq:exterior-trace-direct}
\end{equation}
\end{lemma}
\begin{proof}
For $t$ small enough, every characteristic in the displayed set stays in
$\{\delta/2\le|x|\le R+1\}$ up to time $t$. On this spatial annulus,
\eqref{eq:gain-envelope} gives
\begin{equation}
 \sup_{0<s\le t}\sup_{\delta/2\le|x|\le R+1}
 \|U(s,x)\|_{\infty,m+2}\le C_{m,\delta,R}\eps.
 \label{eq:exterior-local-U}
\end{equation}
Indeed, for $|v|\ge1$ the velocity term in
$\varrho_m^{-(m+2)}$ suffices; for
$|v|<1$, use $|x-sv|\ge\delta/4$. Consequently
$\rho_U\le C_{\delta,R}\eps$ there, and
Lemma~\ref{lem:local-weighted-collision} gives
\[
 \int_0^t\Qp(U,U)(s,y+sv,v)\,\dd s
 \le C_{\delta,R}\eps^2t,
 \qquad A_H(0,t;y,v)\le C_{\delta,R}\eps t
 \quad(0\le H\le U).
\]
Also $f_0(y,v)\le C_\delta\eps$ when $|y|\ge\delta$.
Formula \eqref{eq:gain-only-mild} proves the assertion for $U$.
For $f^{\mathrm{lo}}$ and $f^{\mathrm{up}}$, apply
\eqref{eq:lower-limit-formula}--\eqref{eq:upper-limit-formula}, using
$1-e^{-a}\le a$ for $a\ge0$ and the same gain bound.
\end{proof}
\begin{lemma}[Localization of a positive collision inequality]
\label{lem:stopped-localization}
Fix $0<\delta\le1$ and an integer $n\ge1$. Set
\[
 \Omega_0:=\{x:\delta/8<|x|<2B_m\}.
\]
Let $T>0$. Suppose $H\ge0$ on $(0,T)\times\Omega_0\times\R^3$ has an
absolutely continuous representative along almost every characteristic
segment in this cylinder, with
\begin{equation}
 (\partial_t+v\cdot\nabla_x)H\le\mathcal K_UH,
 \qquad 0\le H\le C_*U,
 \label{eq:localization-hypotheses}
\end{equation}
and zero initial trace on the closed annulus
$\{\delta/8\le|x|\le2B_m\}$, along the characteristics under consideration.
Here $\mathcal K_U$ is defined in \eqref{eq:local-weighted-K}.
Then, for some $0<t_{\delta,n}\le T$,
\begin{equation}
 \esssup_{\substack{0<s\le t,\ \delta/2\le|x|\le B_m\\v\in\R^3}}
 \bracket{v}^{m+2}H(s,x,v)\le C_{m,\delta,n,C_*}t^n,
 \qquad 0<t\le t_{\delta,n}.
 \label{eq:stopped-localization}
\end{equation}
For the characteristic representative, at every fixed
$0<s\leq t_{\delta,n}$,
\[
 \esssup_{\substack{\delta/2\leq|x|\leq B_m\\v\in\R^3}}
 \bracket{v}^{m+2}H(s,x,v)
 \leq C_{m,\delta,n,C_*}s^n.
\]
The constants are uniform for $0<\eps\le\eps_0$, after fixing a sufficiently
small upper bound $\eps_0$ as in Proposition~\ref{prop:gain-only}.
\end{lemma}
\begin{proof}
For $0\le k\le n$, define the open annuli
\[
 \Omega_k:=\left\{x:
 \frac{\delta}{8}+\frac{k\delta}{8n}<|x|
 <2B_m-\frac{kB_m}{2n}\right\}.
\]
This gives the prescribed $\Omega_0$ and
$\Omega_n=\{x:\delta/4<|x|<3B_m/2\}$. Thus
\[
 \{x:\delta/2\le|x|\le B_m\}
 \Subset\Omega_n\Subset\Omega_{n-1}\Subset\cdots\Subset\Omega_1
 \Subset\Omega_0.
\]
The distance between successive boundaries is
\[
 d_k:=\operatorname{dist}(\overline\Omega_k,\partial\Omega_{k-1})
 =\min\left\{\frac{\delta}{8n},\frac{B_m}{2n}\right\}
 =\frac{\delta}{8n}>0,\qquad 1\le k\le n.
\]
The argument for
\eqref{eq:exterior-local-U}, with the inner radius $\delta/8$, gives
\begin{equation}
 \sup_{0<s\le t}\sup_{x\in\Omega_0}\|U(s,x)\|_{\infty,m+2}
 \le C_{m,\delta}\eps
 \label{eq:local-U-weighted}
\end{equation}
when $t$ is sufficiently small. Define
\[
 Z_k(t):=\esssup_{\substack{0<s\le t,\ x\in\Omega_k\\v\in\R^3}}
 \bracket{v}^{m+2}H(s,x,v),\qquad 0\le k\le n.
\]
Then
\begin{equation}
 Z_0(t)\le C_{m,\delta} C_*\eps.
 \label{eq:Z0-bound}
\end{equation}
Fix a characteristic ending at $(s,x,v)$ with $x\in\Omega_k$, and follow
$\gamma(\tau)=x-(s-\tau)v$ backward until time zero or its first exit
from $\Omega_{k-1}$. Denote the stopping time by $\tau_*$. If
$\tau_*>0$, the spatial separation gives
\begin{equation}
 |v|(s-\tau_*)\ge d_k,\qquad |v|\ge d_k/t.
 \label{eq:exit-high-velocity}
\end{equation}
The bound $H\le C_*U$ and Lemma~\ref{lem:rescaled-high-flatness}, with
$R_*=2B_m$ and $N=n$, imply
\begin{equation}
 \bracket{v}^{m+2}H(\tau_*,\gamma(\tau_*),v)
 \le C_{m,\delta,n,C_*}t^n.
 \label{eq:exit-flat-bound}
\end{equation}
The value at the exit time is taken as a limit from inside the annulus:
integrate first from $\tau_*+h$, apply the bound there, and let
$h\downarrow0$. If the characteristic reaches time zero, including an
exit at that time, its initial point lies in the closed annulus where
the initial trace is zero.
Integrating \eqref{eq:localization-hypotheses} from the stopping time to
the final time and applying
Lemma~\ref{lem:local-weighted-collision} and
\eqref{eq:local-U-weighted}, gives
\[
 \bracket{v}^{m+2}H(s,x,v)
 \le C_{m,\delta,n,C_*}t^n
 +C_{m,\delta}\eps\int_{\tau_*}^sZ_{k-1}(\tau)\,\dd\tau.
\]
Hence
\begin{equation}
 Z_k(t)\le C_{m,\delta,n,C_*}t^n+C_{m,\delta}\eps tZ_{k-1}(t),
 \qquad 1\le k\le n.
 \label{eq:nested-annulus-recursion}
\end{equation}
Iterating \eqref{eq:nested-annulus-recursion} and using
\eqref{eq:Z0-bound}, we obtain, for small $t$,
\[
 \begin{aligned}
 Z_n(t)
 \le C_{m,\delta,n,C_*}t^n
 \sum_{j=0}^{n-1}(C_{m,\delta}\eps t)^j
 +(C_{m,\delta}\eps t)^n Z_0(t)\le C_{m,\delta,n,C_*}t^n.
 \end{aligned}
\]
This proves \eqref{eq:stopped-localization}.

For the fixed-time statement, fix $0<s<t$ in the interval just obtained.
By Fubini's theorem, the bound defining $Z_n(t)$ holds for almost every
time along almost every characteristic in $\Omega_n$. Since $H^\#$ is
continuous along these characteristics, the bound also holds at time $s$
whenever $x=y+sv$ lies in the closed annulus of the statement, which is
strictly inside
$\Omega_n$. Therefore
\[
 \esssup_{\substack{\delta/2\le|x|\le B_m\\v\in\R^3}}
 \bracket{v}^{m+2}H(s,x,v)
 \le Z_n(t)\le C_{m,\delta,n,C_*}t^n.
\]
Letting $t\downarrow s$ gives the required bound at every fixed time,
after reducing $t_{\delta,n}$ if necessary.
\end{proof}
\begin{proposition}[High-order short-time gap bound]
\label{prop:finite-gap-main}
For every integer $N\geq1$, there exist $C_{m,N},\tau_{m,N}>0$ such that
\begin{equation}
 \esssup_{(y,v)\in\mathcal A_m}w^\#(t,y,v)\le C_{m,N}t^N,
 \qquad 0<t\le \tau_{m,N}.
 \label{eq:finite-gap-main}
\end{equation}
\end{proposition}
\begin{proof}
Choose an integer $K>mN$. Apply Lemma~\ref{lem:inner-flatness-main} with
$N=K$, and decrease $\delta:=\delta_{m,K}$ to at most one. On the inner part
of $\mathcal A_m$, the bound $0\le w\le U$ gives
$w^\#=O(t^{K/m})=O(t^N)$.
On the exterior, subtract the two characteristic equations in
\eqref{eq:KS-coupled-limits-main}. Their source and loss terms are locally
integrable in the spatial annulus under consideration by
\eqref{eq:local-U-weighted} and
Lemma~\ref{lem:local-weighted-collision}. Thus $w$ is absolutely continuous
along the relevant characteristic segments and
\begin{align}
 (\partial_t+v\cdot\nabla_x)w
 &=\Qp(w,f^{\mathrm{up}})+\Qp(f^{\mathrm{lo}},w)
   -\nu_{f^{\mathrm{lo}}}w+\nu_wf^{\mathrm{lo}}\notag\\
 &\le\Qp(w,U)+\Qp(U,w)+\beta_{\mathbf{b}}\rho_wU
 =\mathcal K_Uw.
 \label{eq:positive-gap-inequality}
\end{align}
Here we used $w\ge0$ and
$0\le f^{\mathrm{lo}},f^{\mathrm{up}}\le U$. The difference has zero initial trace on
the closed annulus by
Lemma~\ref{lem:exterior-trace-direct}, applied on each bounded velocity set.
Lemma~\ref{lem:stopped-localization} with $H=w$, $C_*=1$, and $n=N$
therefore applies. If $(y,v)\in\mathcal A_m$ and $|y|\ge\delta$, then
$|y|\le2^m$, $|v|\le2$, and
$\delta/2\le|y+tv|\le B_m$ for small $t$.
This yields $w^\#\le C_{m,\delta,N}t^N$ on the exterior. Combining the two
regions proves the assertion.
\end{proof}
\section{Equality of the bracket limits and the main theorem}
\label{sec:closure}
The short-time estimate and self-similarity imply that
$w=f^{\mathrm{up}}-f^{\mathrm{lo}}\ge0$ has finite mass and a finite first
velocity moment. We can therefore integrate its equation over space and
velocity. This gives an upper bound for the growth of its mass, while
self-similarity gives an exact growth rate. For small $\eps$, the two are
compatible only if the mass of $w$ is zero. Hence
$f^{\mathrm{up}}=f^{\mathrm{lo}}$.

Fix an integer
\begin{equation}
 N_m>\frac{2(m+1)}{m-1}.
\end{equation}
Self-similarity of $f^{\mathrm{lo}}$ and $f^{\mathrm{up}}$ gives, with
\[
 \widetilde W(y,v):=w^\#(1,y,v),
\]
the exact formula
\begin{equation}
 w^\#(t,y,v)=t^{-(m+2)/(m-1)}
 \widetilde W\bigl(t^{-m/(m-1)}y,t^{-1/(m-1)}v\bigr).
 \label{eq:gap-selfsimilar-main}
\end{equation}
\begin{lemma}[Finite gap mass and absolute first velocity moment]
\label{lem:finite-gap-mass-main}
For every $t>0$,
\begin{equation}
 \int_{\R^6}(1+|v|)w(t,x,v)\,\dd x\dd v<\infty.
 \label{eq:finite-gap-mass-main}
\end{equation}
\end{lemma}
\begin{proof}
Take a fixed dyadic scale
$r=2^j\ge R_1:=\max\{1,\tau_{m,N_m}^{-1/(m-1)}\}$ and
$(y,v)\in\mathcal A_m$. With $t=r^{-(m-1)}$, equations
\eqref{eq:gap-selfsimilar-main} and \eqref{eq:finite-gap-main} give,
for almost every $(y,v)\in\mathcal A_m$,
\[
 \widetilde W(r^my,rv)
 =r^{-(m+2)}w^\#(r^{-(m-1)},y,v)
 \leq C_{m,N_m}r^{-m-2-(m-1)N_m}.
\]
The dilated annuli cover the far field. Since only countably many fixed
scales are used, these estimates give
$\widetilde W\lesssim
\varrho_m^{-m-2-(m-1)N_m}$ almost everywhere there.
On the remaining bounded annulus $1\le\varrho_m\le2R_1$, use the original
pointwise bound and
\[
 \varrho_m^{-(m+2)}
 \leq(2R_1)^{(m-1)N_m}
 \varrho_m^{-m-2-(m-1)N_m}.
\]
After enlarging the constant, the same decay holds for $\varrho_m\ge1$.
For $\varrho_m\leq1$, use $w\leq U$ and \eqref{eq:gain-envelope} to obtain
$\widetilde W\lesssim\eps\varrho_m^{-(m+2)}$.
The volume scaling for $(y,v)\mapsto(r^my,rv)$ has exponent $3m+3$.
Thus an anisotropic annulus of radius $r$ has volume comparable to
$r^{3m+3}$.  Since $|v|\leq\varrho_m(y,v)$,
\begin{align*}
 \int_0^1r^{3m+2}r^{-(m+2)}\,\dd r&<\infty,\\
 \int_1^\infty r^{3m+2}r^{-m-2-(m-1)N_m}\,\dd r&<\infty,\\
 \int_1^\infty r^{3m+2}r\,r^{-m-2-(m-1)N_m}\,\dd r&<\infty.
\end{align*}
The first integral is finite because $m>1$, and the last two are finite by
the choice of $N_m$. This proves the assertion at time one; scaling
gives every $t>0$.
\end{proof}
\begin{lemma}[Distributional balance for the bracket gap]
\label{lem:gap-balance-justification}
On $(0,\infty)\times\R^3_x$, one has in distributions
\begin{equation}
 \partial_t\rho_w+\nabla_x\cdot j_w
 =\beta_{\mathbf{b}}
  (\rho_{f^{\mathrm{up}}}+\rho_{f^{\mathrm{lo}}})\rho_w,
 \qquad
 j_w:=\int_{\R^3}vw\,\dd v.
 \label{eq:gap-density-balance-main}
\end{equation}
Moreover, for every $0<\tau<T<\infty$,
\begin{equation}
 \int_\tau^T\int_{\R^3}
 \bigl((\rho_{f^{\mathrm{up}}}+\rho_{f^{\mathrm{lo}}})\rho_w+|j_w|\bigr)
 \,\dd x\dd t<\infty.
 \label{eq:gap-balance-integrability}
\end{equation}
\end{lemma}
\begin{proof}
On compact time intervals away from zero,
$f^{\mathrm{lo}},f^{\mathrm{up}}\leq U$ and
Lemma~\ref{lem:density-envelope} give finite local velocity densities.  The
integrated gain and loss terms are locally integrable because they are
bounded by a constant multiple of $\rho_U^2$.  Testing the characteristic
formulas therefore shows that the two equations in
\eqref{eq:KS-coupled-limits-main} also hold in distributions.  Subtracting
them gives
\begin{equation}
 (\partial_t+v\cdot\nabla_x)w
 =\Qp(w,f^{\mathrm{up}})+\Qp(f^{\mathrm{lo}},w)
  -\nu_{f^{\mathrm{lo}}}w+\nu_wf^{\mathrm{lo}}.
 \label{eq:gap-equation-main}
\end{equation}
For nonnegative $F,G$ with finite velocity densities, Tonelli's theorem and
the pre--post-collisional change of variables give
\begin{equation}
 \int_{\R^3}\Qp(F,G)(v)\,\dd v
 =\beta_{\mathbf{b}}\rho_F\rho_G.
 \label{eq:gain-mass-identity}
\end{equation}
In center-of-mass and relative-velocity variables, this change
interchanges the two unit directions $\widehat{v-v_*}$ and $\sigma$; it
preserves the measure and leaves
$\mathbf{b}(\widehat{v-v_*}\cdot\sigma)$ unchanged.

We justify the velocity integration in \eqref{eq:gap-equation-main}.
Fix $0<\tau<T$. Define
$M(t)=\int_{\R^6}w(t,x,v)\,\dd x\dd v$.
Lemma~\ref{lem:finite-gap-mass-main} and
\eqref{eq:gap-selfsimilar-main} give
\[
 M(t)=t^{(2m+1)/(m-1)}M(1),
 \qquad
 \int_{\R^6}|v|w(t,x,v)\,\dd x\dd v
 =t^{(2m+2)/(m-1)}
 \int_{\R^6}|V|\widetilde W(X,V)\,\dd X\dd V.
\]
Thus $(1+|v|)w$ is integrable over $[\tau,T]\times\R^6$.
These scaling identities use only the already established
self-similarity and finite moments.

Write the right-hand side of \eqref{eq:gap-equation-main} as
$S=\Qp(w,f^{\mathrm{up}})+\Qp(f^{\mathrm{lo}},w)
-\nu_{f^{\mathrm{lo}}}w+\nu_wf^{\mathrm{lo}}$.
Positivity and \eqref{eq:gain-mass-identity} yield
\[
 \int_{\R^3}|S(t,x,v)|\,\dd v
 \leq\beta_{\mathbf{b}}
 (\rho_{f^{\mathrm{up}}}+3\rho_{f^{\mathrm{lo}}})\rho_w
 \leq4\beta_{\mathbf{b}}\rho_U\rho_w.
\]
Since $\rho_U(t,x)\leq A_{\rho,m}\eps/t$ by
\eqref{eq:density-envelope}, it follows that
\[
 \int_\tau^T\int_{\R^6}|S|\,\dd x\dd v\dd t
 \leq4\beta_{\mathbf{b}}A_{\rho,m}\eps
 \int_\tau^T t^{-1}M(t)\,\dd t<\infty.
\]

Choose $\zeta\in C_c^\infty(\R^3)$ with $0\leq\zeta\leq1$ and
$\zeta=1$ on the unit ball, and put $\zeta_L(v)=\zeta(v/L)$.
For $\psi\in C_c^\infty((\tau,T)\times\R^3_x)$, test
\eqref{eq:gap-equation-main} against $\psi(t,x)\zeta_L(v)$:
\[
 -\int_\tau^T\int_{\R^6}
 w\zeta_L(\partial_t\psi+v\cdot\nabla_x\psi)\,\dd x\dd v\dd t
 =\int_\tau^T\int_{\R^6}S\psi\zeta_L\,\dd x\dd v\dd t.
\]
The integrable bounds for $(1+|v|)w$ and $|S|$ allow $L\to\infty$
by dominated convergence. The transport operator contains no velocity
derivative, so it produces no derivative of $\zeta_L$.
Using \eqref{eq:gain-mass-identity}, the integral of
$-\nu_{f^{\mathrm{lo}}}w$ cancels that of
$\Qp(f^{\mathrm{lo}},w)$, while $\nu_wf^{\mathrm{lo}}$ contributes
$\beta_{\mathbf{b}}\rho_{f^{\mathrm{lo}}}\rho_w$. This proves
\eqref{eq:gap-density-balance-main}.

Finally,
\[
 \int_{\R^3}
 (\rho_{f^{\mathrm{up}}}+\rho_{f^{\mathrm{lo}}})\rho_w\,\dd x
 \leq2A_{\rho,m}\eps t^{-1}M(t),
 \qquad
 \int_{\R^3}|j_w|\,\dd x
 \leq\int_{\R^6}|v|w\,\dd x\dd v.
\]
Both right-hand sides are integrable on $[\tau,T]$, proving
\eqref{eq:gap-balance-integrability}.
\end{proof}
\begin{theorem}[Closure of the bracket]
\label{thm:gap-closure-main}
If, in addition to the smallness required in
Proposition~\ref{prop:gain-only},
\begin{equation}
 2\beta_{\mathbf{b}}A_{\rho,m}\eps<\frac{2m+1}{m-1},
 \label{eq:mass-smallness-main}
\end{equation}
then $f^{\mathrm{lo}}=f^{\mathrm{up}}$ almost everywhere.
\end{theorem}
\begin{proof}
Let $\chi\in C_c^\infty(\R^3)$ satisfy $0\leq\chi\leq1$ and
$\chi=1$ on the unit ball, and put $\chi_R(x)=\chi(x/R)$.  Testing
\eqref{eq:gap-density-balance-main} with $\chi_R$ gives on every compact
time interval $[\tau,T]\Subset(0,\infty)$
\[
 \frac{\dd}{\dd t}\int_{\R^3}\chi_R\rho_w\,\dd x
 =\beta_{\mathbf{b}}\int_{\R^3}\chi_R
 (\rho_{f^{\mathrm{up}}}+\rho_{f^{\mathrm{lo}}})\rho_w\,\dd x
 +\int_{\R^3}\nabla\chi_R\cdot j_w\,\dd x
\]
in distributions in time.  By \eqref{eq:gap-balance-integrability},
\[
 \int_\tau^T\left|\int_{\R^3}\nabla\chi_R\cdot j_w\,\dd x\right|\dd t
 \leq\frac{C}{R}\int_\tau^T\int_{\R^3}|j_w|\,\dd x\dd t
 \longrightarrow0.
\]
Dominated convergence in the remaining terms shows that
\[
 M(t):=\int_{\R^6}w(t,x,v)\,\dd x\dd v
\]
is locally absolutely continuous and
\begin{equation}
 M'(t)=\beta_{\mathbf{b}}\int_{\R^3}
 (\rho_{f^{\mathrm{up}}}+\rho_{f^{\mathrm{lo}}})\rho_w\,\dd x
 \leq\frac{2\beta_{\mathbf{b}}A_{\rho,m}\eps}{t}M(t).
 \label{eq:gap-mass-inequality-main}
\end{equation}
On the other hand, \eqref{eq:gap-selfsimilar-main} and the Jacobian
$\dd x\dd v=t^{(3m+3)/(m-1)}\dd X\dd V$ give
\begin{equation}
 M(t)=t^{(2m+1)/(m-1)}M(1),
 \qquad M'(t)=\frac{2m+1}{(m-1)t}M(t).
 \label{eq:gap-mass-scaling-main}
\end{equation}
Combining \eqref{eq:mass-smallness-main}--
\eqref{eq:gap-mass-scaling-main} forces $M(t)=0$.  Since $w\geq0$, we obtain
$w=0$ almost everywhere.
\end{proof}
\begin{proof}[Proof of Theorem~\ref{thm:main}]
Choose $\eps_0$ so that the contraction in
Proposition~\ref{prop:gain-only} holds for $\eps\le\eps_0$ and
$2\beta_{\mathbf{b}}A_{\rho,m}\eps_0<(2m+1)/(m-1)$.
Proposition~\ref{prop:KS-bracket} and
Theorem~\ref{thm:gap-closure-main} give
$f:=f^{\mathrm{lo}}=f^{\mathrm{up}}$ almost everywhere.
On almost every characteristic, the two limit formulas have the same
coefficients and source and their continuous representatives coincide.
We choose the representative of $f^{\mathrm{lo}}$ fixed above; it satisfies
\eqref{eq:mild-definition} almost everywhere. Since $0\le f\le U$,
the characteristic integrability required in the definition follows
from the bounds proved in Proposition~\ref{prop:KS-bracket}.
The scaling and bounds
follow from Propositions~\ref{prop:gain-only} and~\ref{prop:KS-bracket}
and Lemma~\ref{lem:density-envelope}.
Let $K\Subset\R^6\setminus\{(0,0)\}$. Its anisotropic radius is bounded
above and away from zero. Choose finitely many fixed dilations $r_\nu>0$
and a covering $K=\bigcup_{\nu=1}^LK_\nu$ such that
\begin{equation}
 (y,v)\in K_\nu\quad\Longrightarrow\quad
 \frac12\le\varrho_m(r_\nu^my,r_\nu v)\le2.
 \label{eq:finite-trace-scale-cover}
\end{equation}
For $\delta>0$ sufficiently small there is $a>0$ such that
$|v|\ge a$ on $K\cap\{|y|\le\delta\}$. Decrease $\delta$ so that
$R_0^m\delta<a^m$ and $r_\nu^m\delta\le\delta_{m,1}$ for every $\nu$,
where $\delta_{m,1}$ is given by
Lemma~\ref{lem:inner-flatness-main} with $N=1$. The conic support gives
$f_0=0$ on this inner set. Scaling $U$ by each fixed $r_\nu$ and applying
that lemma gives
\[
 0\le f^\#(t,y,v)\le U^\#(t,y,v)
 \le C_{m,K}\eps^2t^{1/m}
 \quad\text{on }K\cap\{|y|\le\delta\}.
\]
On $K\cap\{|y|\ge\delta\}$, the trace convergence follows directly
from Lemma~\ref{lem:exterior-trace-direct}. Hence
\eqref{eq:main-trace} holds.
Since $\phi\ge0$ and $\phi\not\equiv0$, the datum is strictly positive
on a set of positive measure away from the origin. The integral of the
collision frequency is finite on its free characteristics, and the first term in
\eqref{eq:mild-definition} is positive there. Thus $f\not\equiv0$.
The positive-time infinite-mass assertion is proved in
Proposition~\ref{prop:positive-time-infinite-mass} below.
\end{proof}
\section{The self-similar profile and far trace}
\label{sec:profile}
We write the solution in terms of its profile at time one. We then show
how the initial trace determines the behavior of this profile at infinity
and prove that the solution has infinite mass at every positive time.

Define $F(X,V)=f(1,X,V)$. In free-transport coordinates, the profile is
\[
 \mathcal F(Y,V):=F(Y+V,V)=f^\#(1,Y,V).
\]
\begin{corollary}[Profile formulation]
\label{cor:profile}
The profile $F$ is a nonnegative nonzero distributional solution of
\eqref{eq:profile-equation}.  It satisfies
\begin{equation}
 0\leq\mathcal F(Y,V)\leq C_m\eps\varrho_m(Y,V)^{-(m+2)}
 \label{eq:profile-envelope}
\end{equation}
and the anisotropic far-trace condition
\begin{equation}
 r^{m+2}\mathcal F(r^my,rv)\longrightarrow f_0(y,v)
 \quad(r\to\infty)
 \label{eq:profile-far-trace}
\end{equation}
in $L^\infty_{\mathrm{loc}}(\R^6\setminus\{(0,0)\})$.
\end{corollary}
\begin{proof}
The profile representation \eqref{eq:profile-ansatz} follows from
\eqref{eq:main-selfsimilarity}.  The envelope and density bounds imply that
$f$ and the velocity integrals of its gain and loss terms are locally
integrable; hence the characteristic-mild equation holds distributionally.
Substitution of the profile representation in that distributional equation
gives \eqref{eq:profile-equation}.  The envelope is
\eqref{eq:main-envelope} at time one. Taking $t=r^{-(m-1)}$ in the identity
\[
 f^\#(t,y,v)=t^{-(m+2)/(m-1)}
 \mathcal F\bigl(t^{-m/(m-1)}y,t^{-1/(m-1)}v\bigr)
\]
and using \eqref{eq:main-trace} gives \eqref{eq:profile-far-trace}.
\end{proof}
\begin{proposition}[Infinite mass at positive times]
\label{prop:positive-time-infinite-mass}
For the solution constructed above,
\[
 \int_{\R^6}f(t,x,v)\,\dd x\dd v=\infty
 \qquad\text{for every }t>0.
\]
\end{proposition}
\begin{proof}
Fix $t>0$. On the support of $f_0(y,v)$ one has
$|v|\le R_0|y|^{1/m}$. Thus, for $|y|$ sufficiently large depending on
$t$, the free line satisfies $|y+sv|\ge|y|/2$ for $0\le s\le t$.
Using \eqref{eq:density-envelope} and $f\le U$, we obtain
\[
 A_f(0,t;y,v)\le C_m\eps t|y|^{-(m-1)/m}.
\]
Keeping only the first term in \eqref{eq:mild-definition} gives
\[
 f^\#(t,y,v)\ge e^{-C_m\eps t|y|^{-(m-1)/m}}f_0(y,v)
 \ge\tfrac12 f_0(y,v)
\]
on the sufficiently distant part of this support. On the other hand,
\[
 \rho_{f_0}(y)=\eps|y|^{-(m-1)/m}
 \int_{\R^3}\phi(\widehat y,z)\,\dd z.
\]
By Fubini's theorem, the angular coefficient is positive on a measurable
subset of $\Sph^2$ of positive surface measure, because $\phi\ge0$ and
$\phi\not\equiv0$. The spatial integral of this density diverges at
infinity. Since $(y,v)\mapsto(y+tv,v)$ preserves volume, the assertion
follows.
\end{proof}
\begin{remark}[The role of the Maxwell assumption]
For a cutoff kernel with kinetic factor $|v-v_*|^\gamma$, the loss frequency
depends on $v$. Whenever all integrals are finite, the same collision
symmetries still give
\[
 \int_{\R^3}\Qp(F,G)\,\dd v=\int_{\R^3}F\nu_G\,\dd v
 =\int_{\R^3}G\nu_F\,\dd v.
\]
Consequently the coupled gap balance becomes
\[
 \partial_t\rho_w+\nabla_x\cdot j_w
 =\int_{\R^3}
  (\nu_{f^{\mathrm{up}}}+\nu_{f^{\mathrm{lo}}})w\,\dd v.
\]
The Maxwell assumption reduces this expression to the scalar product in
\eqref{eq:gap-density-balance-main}, which is controlled by the density
estimate already available. For other kinetic factors, a corresponding
argument would require additional control of the velocity-dependent
frequency, or a suitable weighted gap estimate. The necessary gain bounds
and scaling must also be reconsidered; no extension to $\gamma\ne0$ is
asserted here.
\end{remark}
\begin{remark}[Angular cutoff]
The critical gain estimate requires only a bounded nonnegative angular
kernel. The condition $\Lambda_{\mathbf{b}}<\infty$ in
\eqref{eq:strong-cutoff-intro} is used in
Lemma~\ref{lem:angular-integration} and the high-velocity and local weighted
collision bounds. It allows the support of $\mathbf{b}$ to reach both
endpoints; for example, $\mathbf{b}(z)=(1-z^2)^\alpha$ satisfies it when
$\alpha>1/2$. Thus no fixed angular neighborhood of either endpoint has
to be excluded. A general bounded kernel, such as $\mathbf{b}\equiv1$,
need not have finite $\Lambda_{\mathbf{b}}$; the theorem is not asserted
under boundedness alone. A non-cutoff extension would require further
estimates for the angular singularity.
\end{remark}

\section*{Acknowledgements}
Q.-H.~Nguyen was supported by the CAS Project for Young Scientists in
Basic Research (Grant No.~YSBR-031) and the NSFC (Grant Nos.~1251101538
and 12595282). T.~Yang was supported by the General Research Fund of
Hong Kong (Project No.~11302723), a start-up grant from The Hong Kong
Polytechnic University (Project No.~P0043962), and the Kuok Group
Foundation.
Jiaqi Yang is supported by the National Natural Science Foundation of China
(Grant No.~12471225) and the Natural Science Basic Research Program of
Shaanxi (Program No.~2026JC-YXQN-021).
\section*{AI-use disclosure}
The authors used OpenAI's ChatGPT (Sol 5.6 and Astra 6) for language
editing, organization, bibliographic checks, and checking calculations
in arguments they had already developed. The authors reviewed all
AI-assisted revisions, independently verified the mathematical arguments
and proofs, and take full responsibility for the content and conclusions
of the paper.

\bigskip
\begingroup
\small
\begin{samepage}
\noindent\textbf{Quoc-Hung Nguyen}\\
State Key Laboratory of Mathematical Sciences, Academy of Mathematics
and Systems Science, Chinese Academy of Sciences, Beijing 100190, China;\\
Institute of Mathematics, Academy of Mathematics and Systems Science,
the Chinese Academy of Sciences, Beijing 100190, China.\\
\textit{Email address:}
\href{mailto:qhnguyen@amss.ac.cn}{\nolinkurl{qhnguyen@amss.ac.cn}}
\par
\end{samepage}
\medskip
\begin{samepage}
\noindent\textbf{Jiaqi Yang}\\
School of Mathematics and Statistics, Northwestern Polytechnical University.\\
\textit{Email addresses:}
\href{mailto:yjqmath@nwpu.edu.cn}{\nolinkurl{yjqmath@nwpu.edu.cn}},
\href{mailto:yjqmath@163.com}{\nolinkurl{yjqmath@163.com}}
\par
\end{samepage}
\medskip
\begin{samepage}
\noindent\textbf{Tong Yang}\\
Department of Applied Mathematics, The Hong Kong Polytechnic University,
Kowloon, Hong Kong, P. R. China.\\
\textit{Email address:}
\href{mailto:t.yang@polyu.edu.hk}{\nolinkurl{t.yang@polyu.edu.hk}}
\par
\end{samepage}
\endgroup
\end{document}